\documentclass[final, a4paper, 10pt]{article}

\usepackage[scale=0.768]{geometry}
\usepackage[english]{babel}

\usepackage[mathscr]{euscript}
\usepackage{color}
\usepackage{fancyhdr}
\usepackage{bm}
\usepackage{hyperref}
\usepackage{cite}
\usepackage{showkeys}
\usepackage{subfigure}
\usepackage{float}

\usepackage{amsfonts}
\usepackage{amsmath}
\usepackage{amssymb}
\usepackage{amscd}
\usepackage{amsthm}

\usepackage{accents}
\usepackage{graphicx}
\usepackage{array}

\usepackage{algorithm}
\usepackage{algorithmicx}
\usepackage{algpseudocode}

\newtheorem{theorem}{Theorem}
\newtheorem{definition}[theorem]{Definition}
\newtheorem{lemma}[theorem]{Lemma}
\newtheorem{corollary}[theorem]{Corollary}

\newtheorem{remark}[theorem]{Remark}
\newtheorem{assumption}[theorem]{Assumption}

\newcommand*{\N}{\ensuremath{\mathbb{N}}}
\newcommand*{\Z}{\ensuremath{\mathbb{Z}}}
\newcommand*{\R}{\ensuremath{\mathbb{R}}}
\newcommand*{\C}{\ensuremath{\mathbb{C}}}
\renewcommand{\i}{\mathrm{i}}
\renewcommand{\phi}{\varphi}
\renewcommand{\rho}{{\varrho}}
\renewcommand{\epsilon}{{\varepsilon}}

\renewcommand{\d}[1]{\,\mathrm{d}#1 \,}

\newcommand{\D}{\mathcal{D}}

\newcommand{\F}{\mathcal{F}} 

\newcommand{\T}{{\mathcal{T}}}

\newcommand{\B}{{\mathcal{B}}}

\newcommand{\A}{{\mathcal{A}}}
\newcommand{\E}{{\mathcal{E}}}
\renewcommand{\B}{{\mathcal{B}}}
\newcommand{\K}{{\mathcal{K}}}

\renewcommand{\S}{\mathcal{S}}
\newcommand{\UC}{\mathbb{S}^1}
\newcommand{\UD}{\mathbb{D}}
\newcommand{\p}{{per}}
\renewcommand{\L}{\mathcal{L}} 
\renewcommand{\Re}{\mathrm{Re}\,}
\renewcommand{\Im}{\mathrm{Im}\,}

\newcommand{\grad}{\nabla}

\newcommand{\loc}{{\mathrm{loc}}}

\newlength{\dhatheight}

\usepackage{color}
\newcommand{\high}{}

\definecolor{xl}{rgb}{0.8,0.2,0.3}

\usepackage{amsmath}
\usepackage{amsfonts}
\usepackage{amssymb}
\begin{document}

\sloppy

\title{Spectrum decomposition of translation operators in periodic waveguide}
\author{
Ruming Zhang\thanks{Institute for Applied and Numerical mathematics, Karlsruhe Institute of Technology, Karlsruhe, Germany
; \texttt{ruming.zhang@kit.edu}. }}
\date{}
\maketitle

\begin{abstract}
Scattering problems in periodic waveguides are interesting but also challenging topics in mathematics, both theoretically and numerically. As is well known, the unique solvability of these problems is not always guaranteed. To obtain a unique solution that is ``physically meaningful'', the {\em limiting absorption principle (LAP)} is a commonly used method. LAP assumes that the limit of a family of solutions with absorbing media converges, as the absorption parameter tends to $0$, and the limit is the ``physically meaningful solution''. It is also called the {\em LAP solution} in this paper.  It has been proved that the LAP holds for periodic waveguides in \cite{Hoang2011}. 
In this paper, we consider the spectrum decomposition of periodic translation operators. With the curve integral formulation and a generalized Residue theorem, the operator is explicitly described by its eigenvalues and generalized eigenfunctions, which are closely related to Bloch wave solutions. Then the LAP solution is decomposed into generalized eigenfunctions. This gives a better understanding of \high{the} structure of the scattered fields.\\

\end{abstract}

\section{Introduction}

In this paper, we consider  scattering problems in periodic waveguides. This topic is of great interest both in mathematics and related technologies, e.g., nano-technology. In recent years, several mathematicians have been working on this topic and many interesting theoretical results and numerical methods have been established. Different from scattering problems with Dirichlet rough surfaces (see \cite{Kirsc1993,Kirsc1994,Chand1996,Chand2010}), the unique solvability is no longer guaranteed due to the existence of eigenvalues. To find the ``physically meaningfully solution'', the {\em Limiting Absorption Principle (LAP)} is a well-known technique. The LAP has been proved in \cite{Hoang2011} by the study of resolvent of periodic elliptic operators, and we also refer to \cite{Fliss2015}  and \cite{Kirsc2017a} with the help of different methods.  A number of numerical methods have also been developed based on the LAP (see \cite{Joly2006,Fliss2009a,Sun2009,Ehrhardt2009,Ehrhardt2009a}).

This paper considers the spectrum decomposition of the periodic operator that maps the values of LAP solutions  between edges of periodic cells. In \cite{Joly2006}, an open question arose that \high{was} if the translation operator between periodic boundaries has a Jordan normal form, when the medium is absorbing. This leads to an interesting \high{topic  to find} out an explicit structure for the solution. The open question has already been proved by T. Hohage and S. Soussi in \cite{Hohag2013}, for cases when either the media is absorbing, or the scattering problem is uniquely solvable. In this paper, we use a different method to establish the spectrum decomposition for cases when the limiting absorption principle holds. This result is definitely  much more general, as the unique solvability does not hold even for  homogeneous planar waveguides, while the LAP holds for a much larger family of problems. Based on the resolvent and the Floquet-Bloch transform, the LAP solution is written as the combination of \high{a} finite number of eigenfunctions and a curve integral. There are infinite number of poles in the interior of the curve, with one accumulation point at $0$ thus the Residue theorem could not be applied.  We first extend the residue theorem to this case, then the integral is written as the sum of residues at the poles. Then we consider the properties of the residues, where a very nice property of periodic structure is adopted. Finally, we have proved that the $H^{1/2}$ space defined on one periodic edge could be decomposed by generalized eigenfunctions. This result shows that the periodic operator has a Jordan normal form and also describes  structure of LAP solutions.

The rest of this paper is organized as follows. We give a brief \high{introduction} of the scattering problem in the second section, and introduce the spectrum of the quasi-periodic solution operator in Section 3. We formulate variational formulations for the quasi-periodic problems in the fourth section. In Section 5, we introduce the Floquet-Bloch transform and apply it to the case when medium is absorbing. Then we consider the case when the medium is not absorbing with the help of the limiting absorption principle in the sixth section, and set up the integral formulation for this problem. We study the space of LAP solutions in Section 7, and in Section 8, we study the spectrum decomposition of the periodic operator, also decompose the LAP solution by residues of generalized eigenfunctions.

\section{Scattering problems in periodic waveguides}

In this section, we introduce the mathematical model of scattering problems in a periodic waveguide (see Figure \ref{waveguide}). Let the waveguide \high{$\Omega:=(-\infty,\,+\infty)\times(0,1)$} be filled up with a periodic material that is periodic in $x_1$-direction, and the refractive index is denoted by a real-valued function $q$. 
Let the period be one, then
\begin{equation*}
 q(x_1+1,x_2)=q(x_1,x_2),\quad \forall x\in\Omega.
\end{equation*}
We also assume that $q\geq c_0>0$ where $c_0$ is a positive constant. Then we consider the solution of the following scattering problem:
\begin{eqnarray}
 \Delta {u}+k^2 q {u}&=&f\quad\text{ in }{\Omega};\label{eq:wg1}\\
 \frac{\partial {u}}{\partial \nu}&=&0\quad\text{ on }\partial \Omega;\label{eq:wg2}
\end{eqnarray}
where $\nu$ is the normal outward vector,  $f$ is a function in $L^2(\Omega)$ with compact support, the boundary of $\Omega$ is 
\[\partial\Omega:=\high{\big\{(x_1,0):\,x_2\in\R\big\}\cup\big\{(x_1,1):\,x_1\in\R\big\}.}\] 

\begin{remark}
Different boundary conditions on $\partial\Omega$ could also be considered, e.g., Dirichlet boundary conditions, Robin boundary conditions, etc.  The method is similar for these cases, thus we only use the Neumann condition in this paper as an example.
\end{remark}

\begin{figure}[ht]
\centering
\includegraphics[width=16cm,height=3cm]{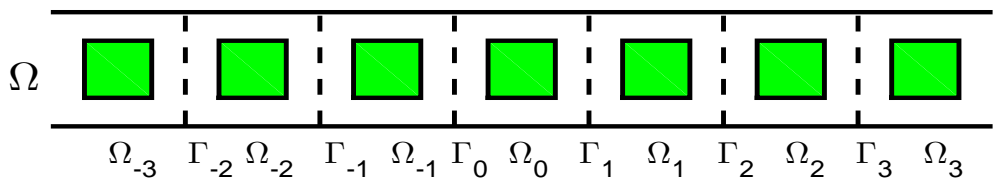}
\caption{Periodic waveguide.}
\label{waveguide}
\end{figure}

We are interested in the behavior of the scattered field $u$ between two periodic cells, thus we define the following periodic domains. Let
\begin{equation*}
 \Omega_j:=(j,j+1]\times(0,1),\quad \Gamma_j:=\{j\}\times(0,1),
\end{equation*}
then the left and right boundary of $\Omega_j$ are $\Gamma_j$ and $\Gamma_{j+1}$. Moreover, $\Omega=\cup_{j\in\Z}\Omega_j$. For simplicity, we also assume that ${\rm supp}(f)\subset \Omega_0$. Moreover, we define 
\[\partial \Omega_j:=\partial\Omega\cap\overline{\Omega_j}=\Big((j,j+1]\times\{0\}\Big)\cup\Big((j,j+1]\times\{1\}\Big),\]
which is the union of the lower and upper boundary of $\Omega_j$.

As is well known, the problem \eqref{eq:wg1}-\eqref{eq:wg2} is not always uniquely solvable in $H^1(\Omega)$, due to the distribution of spectrum of the elliptic operator. To identify ``physically meaningful'' solutions, we adopt the {\em Limiting Absorption Principle (LAP)}. When $k^2$ is replaced by $k_\epsilon^2:=k^2+\i\epsilon$ for any $\epsilon>0$, the problem \eqref{eq:wg1}-\eqref{eq:wg2} has a unique solution $u_\epsilon\in H^1(\Omega)$ from Lax-Milgram theorem. When $\lim_{\epsilon\rightarrow 0^+}u_\epsilon$ exists in $H^1_{loc}(\Omega)$, the LAP holds and the limit, denoted by $u(f)$, is called the LAP solution.

In this paper, we would like to consider the so-called translation operator for LAP solutions. The translation operator is defined as follows. Let $u(f)$ be the LAP solution to \eqref{eq:wg1}-\eqref{eq:wg2}, then the operator is defined
\begin{equation*}
 \mathcal{R}:\, u(f)\big|_{\Gamma_j}\,\mapsto \,u(f)\big|_{\Gamma_{j+1}},\quad j\geq 1.
\end{equation*}
As $u(f)\in H^1_{loc}(\Omega)$, $u(f)\big|_{\Gamma_j}\in H^{1/2}(\Gamma_j)$ for any $j\in\Z$. As $\Gamma_j$ is only a translation of $\Gamma_1$, let $\Gamma:=(0,1)$ and $X:=H^{1/2}(\Gamma)$, then $\mathcal{R}$ is an operator from $X$ to $X$. From \cite{Joly2006}, $\mathcal{R}$ has the following properties:
\begin{itemize}
 \item $\mathcal{R}$ is a compact operator;
 \item the spectrum radius $\rho(\mathcal{R})\leq 1$.
\end{itemize}

\begin{remark}
 We only discuss the translation operator defined on the right of $\Omega_0$, i.e., from $\Gamma_j$ to $\Gamma_{j+1}$ when $j\geq 1$. In fact, the operator on the left, i.e., from $\Gamma_j$ to $\Gamma_{j-1}$ when $j\leq 0$, is similar, thus is omitted.
\end{remark}

To make descriptions more clearly, at the end of this section, we introduce  the following operator:
\begin{equation*}
\A u:=-{q^{-1}}\high{ \Delta u}
\end{equation*}
defined in the domain 
\begin{equation*}
D(\A,\Omega):=\Big\{u\in H^1(\Omega):\,\Delta u\in L^2(\Omega)\Big\}.
\end{equation*}
Then equation \eqref{eq:wg1} is equivalent to 
\[(k^2 I-\A)u=\high{q^{-1}f},\quad \text{ where $I$ is the identity operator}.\] 
Let the spectrum of $\A$ be denoted by $\sigma(\A)$, then the problem \eqref{eq:wg1}-\eqref{eq:wg2} is  uniquely solvable in $H^1(\Omega )$ if and only if $k^2\notin\sigma(\A)$. Thus the spectrum of $\A$ plays an important role in the well-posedness of the problem \eqref{eq:wg1}-\eqref{eq:wg2}. From the Floquet-Bloch theory, the spectrum of $\A$ is closely related to the so-called Bloch wave solutions, which will be introduced in the next section.

\section{\high{Spectral} properties of $\A$}

In this section, we introduce Bloch wave solutions and spectrum properties of the operator $\A$. The mathematical basis mainly comes from the Floquet-Bloch theory in \cite{Kuchm1993}, and for more details of scattering problems we refer to  \cite{Joly2006,Ehrhardt2009a,Fliss2015}.

For a complex number  $z\in\C$, define the $z$-quasi-periodic boundary condition as:
\begin{equation}\label{eq:z_quasi}
\left.u\right|_{\Gamma_{j+1}}=z\, \left.u\right|_{\Gamma_j},\quad  \left.\frac{\partial u}{\partial x_1}\right|_{\Gamma_{j+1}}=z\left.\frac{\partial u}{\partial x_1}\right|_{\Gamma_{j}},\quad \forall j\in\Z.
\end{equation}
We define the subspace of $H^1(\Omega_0)$ by:
\begin{equation*}
H_z^1(\Omega_0):=\Big\{\phi\in H^1(\Omega_0):\,\phi\text{ satisfies \eqref{eq:z_quasi} for $j=0$}\Big\}.
\end{equation*}
\high{In particular}, when $z=1$, the functions that satisfy \eqref{eq:z_quasi} are periodic, and the subspace of periodic functions in $H^1(\Omega_0)$ is denoted by $H_\p^1(\Omega_0)$.
 Define the operator in the quasi-periodic domain, i.e.,
\begin{equation}
\A_z u=-q^{-1}\high{\Delta u}\,\text{ with its domain }D_z(\A,\Omega_0):=D(\A,\Omega_0)\cap H_z^1(\Omega_0),
\end{equation}
where $D(\A,\Omega_0)$ is defined in the same way as $D(\A,\Omega)$, with $\Omega$ replaced by $\Omega_0$. Let $\sigma(\A_z)$ be the spectrum of $\A_z$. A classical result from the Floquet-Bloch theory also shows that (see \cite{Kuchm1993}):
\begin{equation}\label{eq:A_sig_A}
\sigma(\A)=\cup_{|z|=1}\sigma(\A_z).
\end{equation}

 When  the equation $(k^2 I-\A_z)u=0$ has a nontrivial solution in $D_z(\A,\Omega_0)$,   $z$ is called a {\em Floquet multiplier}, and the corresponding non-trivial solution is called a {\em Bloch wave solution}. Let $\mathbb{F}\high{(k^2)}$ be the collection of all Floquet multipliers and  ${\mathbb U\mathbb F}\high{(k^2)}=\mathbb{F}\high{(k^2)}\cap\UC$ ($\UC$ is the unit circle in $\C$) be the set of all {\em unit Floquet multipliers}.  
 
\high{
\begin{remark}
In this paper, when $k^2$ is fixed, we write $\mathbb{F} \,(\mathbb{UF})$ instead of $\mathbb{F}(k^2) \,(\mathbb{UF}(k^2))$ for simplicity.
\end{remark}

From \cite{Joly2006}, we have the following property for $\mathbb{F}$:\\
{\bf Property a). } 
 $z\in\mathbb{F}$ if and only if $z^{-1}\in\mathbb{F}$. Especially, $z\in\mathbb{UF}$ if and only if $\overline{z}=z^{-1}\in\mathbb{UF}$.

From the proofs of Lemma 2.3 and Theorem 2.4 (a) in \cite{Kirsc2017a}, we also have the following property:
 \\
{\bf Property b). } 
 $\mathbb{F}$ is a discrete set with only accumulation points at $0$ and $\infty$.
 
 This property is also shown later in Theorem \ref{th:z-eq}.
 }

From the \high{Floquet-Bloch theory} \eqref{eq:A_sig_A}, to study the \high{spectral} property of $\A$, it is \high{useful} to consider the spectrum of the operator $\A_z$ for the cell problem when $|z|=1$. For simplicity, we replace $z$ by $\alpha=-\i\log(z)$ where $\alpha\in(-\pi,\pi]$ and also replace $\A_z$ by $\A_\alpha$ in this section. Then \eqref{eq:z_quasi} becomes
\begin{equation}
\label{eq:alpha_quasi}
\left.u\right|_{\Gamma_{j+1}}=\exp(\i\alpha) \left.u\right|_{\Gamma_j},\quad  \left.\frac{\partial u}{\partial x_1}\right|_{\Gamma_{j+1}}=\exp(\i\alpha) \left.\frac{\partial u}{\partial x_1}\right|_{\Gamma_{j}},\quad \forall j\in\Z.
\end{equation}
Denote the spectrum of $\A_\alpha$ by $\sigma(\A_\alpha)$. \high{As $\A_\alpha$ is self-adjoint with respect to the $L^2-$space equiped with the weighted inner product $(\phi,\psi)_{L^2,q}=\int_{\Omega_0}q\phi\overline{\psi}\d x$}, $\sigma(\A_\alpha)$ is a discrete subset of $(0,\infty)$ (see \cite{Kuchm1993,Fliss2015}). 
 By rearranging the order of the points in $\sigma(\A_\alpha)$ properly, we obtain a family of analytic functions $\{\mu_n(\alpha):\,n\in\N\}$ such that
 \begin{equation*}
 \sigma(\A_\alpha)=\cup_{n\in\N}\{\mu_n(\alpha)\}.
 \end{equation*}
 Thus $\sigma(\A)=\cup_{n\in\N}\cup_{\alpha\in(-\pi,\pi]}\{\mu_n(\alpha)\}$.  
Moreover, $\lim_{n\rightarrow\infty}\mu_n(\alpha)=\infty$ for any fixed $\alpha\in(-\pi,\pi]$.
For any $\mu_n(\alpha)$, there is also a corresponding eigenfunction $\psi_n(\alpha,\cdot)\in D_\alpha(\A,\Omega_0)$ such that
\begin{equation*}
\A_\alpha\psi_n(\alpha,\cdot)=\mu_n(\alpha)\psi_n(\alpha,\cdot).
\end{equation*}
 Moreover, $\psi_n(\alpha,\cdot)$ also depends analytically on $\alpha\in(-\pi,\pi]$. Both $\mu_n(\alpha)$ and $\psi_n(\alpha,\cdot)$ can be extended into analytic functions with respect to $\alpha$ in a small enough neighborhood of $(-\pi,\pi]\times\{0\}$ in the complex plane.

\high{
\begin{remark}
From Section 3.3 in \cite{Fliss2015}, there is an important argument for the function $\mu_n$. As $\mu_n$ is analytic, either $\mu_n$ is constant or $\mu'_n$ vanishes at \high{a} finite number of points. Note that the function $\mu_n$ is a constant function if and only if $\A$ has a non-empty point spectrum (denoted by $\sigma_p(\A)$). However, as was proved in Theorem 2.4 (b) in \cite{Kirsc2017a}, $\sigma_p(\A)=\emptyset$. Thus $\mu_n$ is never a constant function.
\end{remark}

From the remark above, we also have a further property for $\mathbb{UF}$.\\
{\bf Property c). }The set $\mathbb{UF}$ is either finite or empty.

}

For any fixed $n\in\N$, the graph $\{(\alpha,\mu_n(\alpha)):\,\alpha\in(-\pi,\pi]\}$ is called a dispersion curve, and all the dispersion curves compose the dispersion diagram. We show the dispersion diagram \high{in} the following example:\\
\noindent
{\bf Example}. $q=1$ is a constant function in $\Omega$, and its dispersion diagram is shown in Figure \ref{fig:dd} (left). The relationship between the quasi-periodic parameter $\alpha$ and the eigenvalue \high{is simply}:
\begin{equation*}
\mu_{jm}(\alpha)=j^2\pi^2+(\alpha+2\pi m)^2,\quad j\in\N,\,m\in\Z.
\end{equation*}

\begin{figure}[ht]
\centering
\includegraphics[width=0.45\textwidth]{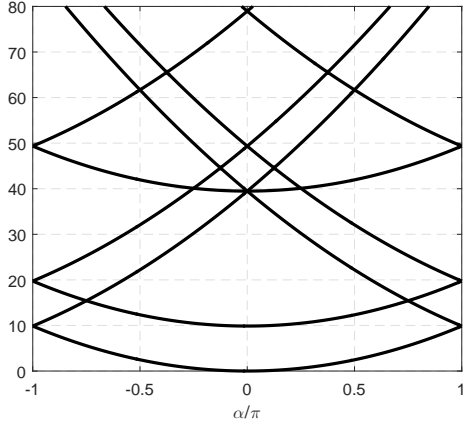} 
\caption{Dispersion diagram.}
\label{fig:dd}
\end{figure}


 For any fixed $k^2\in\sigma(A)$,  there is at least one $\alpha\in(-\pi,\pi]$ such that $k^2\in\sigma(A_\alpha)$. Thus the set
 \begin{equation*}
  P :=\left\{\alpha\in(-\pi,\pi]:\,\exists\, n\in\N,\,{\rm s.t.,}\, \mu_n(\alpha)=k^2\right\}
 \end{equation*}
is not empty. From the definition of $\mathbb{UF}$, it has the representation as:
\begin{equation*}
 \mathbb{UF}=\left\{\exp(\i\alpha):\,\alpha\in P\right\}.
\end{equation*}
The points in $P $ \high{are} divided into the following three classes:
\begin{itemize}
 \item  When $\mu'_n(\alpha)>0$, then $\psi_n(\alpha,\cdot)$ is propagating from the left to the right;
 \item when $\mu'_n(\alpha)<0$, then $\psi_n(\alpha,\cdot)$ is propagating from the right to the left;
 \item when $\mu'_n(\alpha)=0$.
\end{itemize}
Based on \high{this classification}, we define the following three sets:
\begin{eqnarray*}
 P_\pm &:=&\left\{\alpha\in(-\pi,\pi]:\,\exists\, n\in\N\text{ s.t., }\mu_n(\alpha)=k^2\text{ and }\pm\mu'_n(\alpha)>0\right\};\\
 P_0 &:=&\left\{\alpha\in(-\pi,\pi]:\,\exists\, n\in\N\text{ s.t., }\mu_n(\alpha)=k^2\text{ and }\mu'_n(\alpha)=0\right\}.
\end{eqnarray*}
Then $P =P_+ \cup P_- \cup P_0 $.

\begin{remark}
It is possible that there are more than one different integers $m_1,\dots,m_M\in\N$ such that $\mu_{m_1}(\alpha)=\cdots=\mu_{m_M}(\alpha)=k^2$, i.e., $M$ dispersion curves may have an intersection at $(\alpha,k^2)$. \high{In this case, $\alpha$ is repeated $M$ times, i.e., $\left\{\alpha_{m_1},\dots,\alpha_{m_M}\right\}$, and each $\alpha_{m_j}$ is associated with a unique dispersion curve $\mu_{m_j}$ ($j=1,2,\dots,M$).}
\end{remark}

%
%

As the limiting absorption principle fails when the set $P_0 $ is not empty, we make the following assumption.

\begin{assumption}\label{asp1}
 Assume \high{that  $k$ is such that $P_0 =\emptyset$.}
\end{assumption}
The assumption is reasonable as the set $\big\{k>0:\,P_0 \neq\emptyset\big\}$ is ``small enough''. Actually, the set is countable with at most one accumulation point at $\infty$ (see Theorem 5, \cite{Fliss2015}). As was proved in \cite{Hoang2011,Fliss2015,Kirsc2017a}, the limiting absorption principle holds when Assumption \ref{asp1} holds.

\section{Variational form of quasi-periodic solutions}

From the last section, the Floquet multipliers play an important role in the spectrum of the operator $\A$. Thus we introduce the quasi-periodic solutions in this section. 
\high{Consider variational formulation for the following quasi-periodic problem. For any $z\in\C\setminus\{0\}$ with right hand side $f_z\in L^2(\Omega_0)$ that depends analytically on $z$, find a weak solution $u_z\in H^1_z(\Omega_0)$ of the following problem:}
\begin{eqnarray}\label{eq:quasi1}
 \Delta u_z+k^2 qu_z&=&f_z\quad\text{ in }\Omega_0;\\\label{eq:quasi2}
 \frac{\partial u_z}{\partial \nu}&=&0\quad\text{ on }\partial\,{\Omega_0}.
\end{eqnarray}
Note that the analytic dependence on $z$ is defined as follows.
\begin{definition}\label{def}
Suppose $\phi(z,\cdot)\in S(X)$ for any fixed $z$, where $S(X)$ is a Sobolev space defined in the domain $X$. Then $\phi$ depends analytically on $z$ in an open domain $U\subset\C$ if for any fixed $z_0\in U$, the expansion
 \begin{equation*}
  \phi(z,x)=\sum_{\ell=0}^\infty (z-z_0)^\ell\phi_\ell(x),\quad\phi_\ell\in S(X)
 \end{equation*}
holds uniformly for $z\in B(z_0,\delta)\subset U$ for a small enough $\delta>0$.
\end{definition}

We seek for the \high{weak} solution  $u_z$ of \eqref{eq:quasi1}-\eqref{eq:quasi2} in the domain $ H^1_z(\Omega_0)$. First, we transform the problem into one defined in the fixed function space $H_\p^1(\Omega_0)$. Define the operator 
\begin{equation*}
 \left(\zeta_z u_z\right)(x_1,{x_2}):=z^{x_1}u_z(x_1,{x_2}),
\end{equation*}
and let $v_z=\zeta_z^{-1} u_z=z^{-x_1}u_z(x_1,{x_2})$. Then $v_z\in H_\p^1(\Omega_0)$ is a periodic \high{weak solution of the following equations:}
\begin{eqnarray*}
 \Delta v_z+2\log(z)\frac{\partial v_z}{\partial x_1}+(k^2q+\log^2(z))v_z&=&\zeta_z^{-1} f_z\quad\text{ in }\Omega_0;\\
 \frac{\partial v_z}{\partial \nu}&=&0\quad\text{ on }\partial\,{\Omega_0}.
\end{eqnarray*}
As $\log(z)$ is a multi-valued function, we require that $z$ lies in the branch cutting off along the negative real axis (denoted by $\R_{\leq 0}:=(-\infty,\,0]\times\{0\}\subset\C$) such that $\log(z)$ is a single valued analytic function in this branch. More explicitly, let the branch be defined as $\C_\times:=\{z\in\C\setminus\{0\}:\,-\pi<\arg(z)\leq \pi\}$, where $\arg(z)$ is the argument of the complex number $z$. 
From the Green's formula, we obtain the variational formulation of the periodic problem, i.e., given an $f_z\in L^2(\Omega_0)$, to find $v_z\in H_\p^1(\Omega_0)$ that satisfies
\begin{equation}\label{eq:quasi_periodic_var}
\int_{\Omega_0}\left[\grad v_z\cdot\grad\overline{\phi}+\log(z)\left(v_z\frac{\partial\overline{\phi}}{\partial x_1}-\frac{\partial v_z}{\partial x_1}\overline{\phi}\right)-(k^2 q+\log^2(z))v_z\overline{\phi}\right]\d x\\
=-\int_{\Omega_0}\left(\zeta_z^{-1} f_z\right)\overline{\phi}\d x
\end{equation}
for any $\phi\in H_\p^1(\Omega_0)$. 
The left hand side is a \high{Hermitian} sesquilinear form defined in $H_\p^1(\Omega_0)\times H_\p^1(\Omega_0)$.
 From Riesz's lemma, there is a bounded linear operator $\K_z\in\mathcal{L}(H_\p^1(\Omega_0))$ such that
\high{\begin{equation*}
\left<\K_z w,\phi\right>= -\int_{\Omega_0}\left[\log(z)\left(w\frac{\partial\overline{\phi}}{\partial x_1}-\frac{\partial w}{\partial x_1}\overline{\phi}\right)-(k^2 q+1+\log^2(z))w\overline{\phi}\right]\d x,
\end{equation*}}
where $\left<\cdot,\,\cdot\right>$ is the inner product defined in $H_\p^1(\Omega_0)$. Moreover, \high{from Riesz representation theorem,} there is an $\widetilde{f}_{z}\in H_z^1(\Omega_0)$ such that
\begin{equation}\label{eq:tilde_f}
-\int_{\Omega_0}\zeta_z^{-1}f_z\overline{\phi_z}\d x=\left<\zeta_z^{-1}\widetilde{f}_{z},\phi_z\right>,\quad\text{ for all }\phi\in H_z^1(\Omega_0).
\end{equation}
Thus the variational form  \eqref{eq:quasi_periodic_var} is equivalent to 
\begin{equation*}
(I-\K_z)v_z=\zeta_z^{-1}\widetilde{f}_z.
\end{equation*}
Thus when $I-\K_z$ is invertible, 
\begin{equation*}
u_z=(I-\B_z)^{-1}\widetilde{f}_z,\quad\text{ where }\B_z=\zeta_z \K_z \zeta_z^{-1}.
\end{equation*}

Now we focus on the invertibility of the operator of $I-\K_z$. \high{From the compact imbedding of $H^1_z(\Omega_0)$ in $L^2(\Omega_0)$, $\K_z$ is a compact operator (the same argument was made in page 3958, \cite{Kirsc2017a} for the compactness of the operator $K_\alpha$ in that paper). As $\K_z$} depends analytically on $z\in\C_\times$, thus the left hand side in \eqref{eq:quasi_periodic_var} defines an analytic family of Fredholm operators $I-\K_z$. First, we recall the following {\em Analytic Fredholm Theory}.

\begin{theorem}[Theorem VI.14, \cite{Reed1980}]
 Let $D$ be an open and connected subset in $\C$, $X$ be a Hilbert space, and $\T:\,D\rightarrow \mathcal{L}(X)$ be an operator valued analytic function such that $\T(z)$ is compact for each $z\in D$. Then either
 \begin{itemize}
  \item $(I-\T(z))^{-1}$ does not exist for any $z\in D$, or
  \item $(I-\T(z))^{-1}$ exists for all $z\in D\setminus S$, where $S$ is a discrete subset of $D$. In this case, $(I-\T(z))^{-1}$ is meromorphic in $D$ and analytic in $D\setminus S$. The residues at the poles are finite rank operators, and if $z\in S$, then $\T(z)\phi=\phi$ has a nonzero solution in $X$.
 \end{itemize}
 \label{th:ana_fred_thy}
\end{theorem}



With the help of the analytic Fredholm theory, 
the existence and regularity of the inverse of $I-\K_z$ with respect to $z\in \C_\times$ is concluded in the following theorem.

\begin{theorem}\label{th:z-eq}
 For any fixed $k\in\C$ with $\Re (k)>0$ and $\Im(k) \geq 0$,  $I-\K_z$ is invertible for  $z\in\C_\times\setminus\mathbb{F}$, where $\mathbb{F}$ is the set of all Floquet multipliers. The inverse operator depends analytically on $z$ in $\C_\times\setminus\mathbb{F}$ and meromorphically on $z$ in $C_\times$.
\end{theorem}

\begin{proof}
 From the fact that $\mathbb{UF}$ is a finite or empty set, $(I-\K_z)^{-1}$ exists for all $z\in\UC\setminus\mathbb{UF}$. From the analytic Fredholm theory, there is a discrete set $\S\subset \C_\times$ \high{depending} on $k$ such that $I-\K_z$ is invertible \high{outside $\S$}.  The analytic dependence of $z$ comes from the perturbation theory \high{(see \cite{Kato1995})}.

We only need to prove that $\S=\mathbb{F}$. Suppose there is a $z\in \S\setminus\mathbb{F}$, then with $f=0$, there is a nontrivial solution $\phi\in H_\p^1(\Omega_0)$ such that $(I-\K_z)\phi=0$. Then $\zeta_z\phi\in H_z^1(\Omega_0)$ is a nontrivial solution to \eqref{eq:quasi1}-\eqref{eq:quasi2} with $f_z=0$, thus $z\in\mathbb{F}$. It contradicts with the assumption that  $z\in\S\setminus\mathbb{F}$. Suppose there is a $z\in \mathbb{F}\setminus\S$, then there is  a $\psi\in H_z^1(\Omega_0)$ that satisfies \eqref{eq:quasi1}-\eqref{eq:quasi2} with $f_z=0$. Then $\zeta_z^{-1}\psi\in H^1_\p(\Omega_0)$ is a solution to \eqref{eq:quasi_periodic_var} with vanishing right hand side. This implies $\zeta_z^{-1}\psi$ satisfies $(I-\K_z)(\zeta_z^{-1}\psi)=0$, so $z\in \S$, which contradicts with the assumption that $z\in\mathbb{F}\setminus \S$.  Thus $\S=\mathbb{F}$. 
\end{proof}

\high{
Note that $\mathbb{F}(k^2)=\S(k^2)$ is the collection of all poles of the operator $\left(I-\K_z(k^2)\right)^{-1}$, where $\K_z(k^2)$ depends analytically on $z$ and jointly continuous on $(z,k^2)$. Thus we also get the following property of $\mathbb{F}(k^2)$ from Theorem 3 in \cite{Stein1968}:\\
{\bf Property d). }The set $\mathbb{F}(k^2)$ depends continuously on $k^2$.

Moreover, as the set $\S(k^2)$ is discrete, $\mathbb{F}(k^2)$ is discrete as well. This coincides with Property b).
}

%
%
%

 From the definition $\B_z:=\zeta_z\K_z\zeta_z^{-1}$ and $I-\B_z=\zeta_z(I-\K_z)\zeta_z^{-1}$, the singularities of $(I-\B_z)^{-1}=\zeta_z(I-\K_z)^{-1}\zeta_z^{-1}$ only come from $(I-\K_z)^{-1}$ for $z\neq 0$. On the other hand, as the problem \eqref{eq:wg_z1}-\eqref{eq:wg_z2} has a unique solution when $z\high{\in}\C\setminus\mathbb{F}$, $u_z$ is a single-valued function. This implies that the value of $u_z$ does not depend on the branch where $z$ lies in. Thus it \high{can} be extended to an analytic function in $\C\setminus(\mathbb{F}\cup\{0\})$. 
Then the dependence of the solution $u_z$ of the problem \eqref{eq:quasi1}-\eqref{eq:quasi2} on $z$ \high{is summarized} in the following theorem.

\begin{theorem}\label{th:solution_z_uc}
For any $k\in\C$ such that $\Re(k)>0$ and $\Im(k)\geq 0$ and $f_z\in L^2(\Omega_0)$ depends analytically on $z\in\C$, \eqref{eq:quasi1}-\eqref{eq:quasi2} is uniquely solvable in $H^1_z(\Omega_0)$.  Moreover, $u_z$  depends analytically on $z$ in $\C\setminus\left(\mathbb{F}\cup\{0\}\right)$ and meromorphically on $z$ in $\C\setminus\{0\}$.
\end{theorem}

Note that $0$ is special as it is the branch point of $\zeta_z$. At this point, neither $\K_z$ or $\B_z$ is well-defined. As $u_z$ depends meromorphically on $z\in \C\setminus\{0\}$, no matter $0$ lies in $\mathbb{F}$ or not, it is the only accumulation point of the set $\mathbb{F}$ except for the infinity \high{(see Property b))}.

\high{

At the end of this section, we discuss the distribution of the set $\mathbb{F}=\mathbb{F}(k^2)$. When $k^2\in\sigma(\A)$, the set $\mathbb{UF}$ is not empty. We define three subsets of $\mathbb{UF}$ from the definitions of $P_\pm $ and $P_0 $ in the last section by:
\begin{equation*}
S_\pm^0:=\left\{z=\exp(\i\alpha):\,\alpha\in P_\pm \right\},\quad S_0^0:=\left\{z=\exp(\i\alpha):\,\alpha\in P_0 \right\}.
\end{equation*} 
When Assumption \ref{asp1} is satisfied, $S_0^0=\emptyset$, then $\mathbb{UF}=S_+^0\cup S_-^0$. From the definitions of $P_\pm $, when $z\in S_+^0$, the corresponding propagating Floquet mode is propagating to the right; while when $z\in S_-^0$, the corresponding propagating Floquet mode is propagating to the left. See Figure \ref{fig:z_alpha} for the unit Floquet multipliers in both $\alpha$- and $z$-space. 
 
 \begin{figure}[ht]
\centering
\begin{tabular}{c  c}
\includegraphics[width=0.4\textwidth,height=0.36\textwidth]{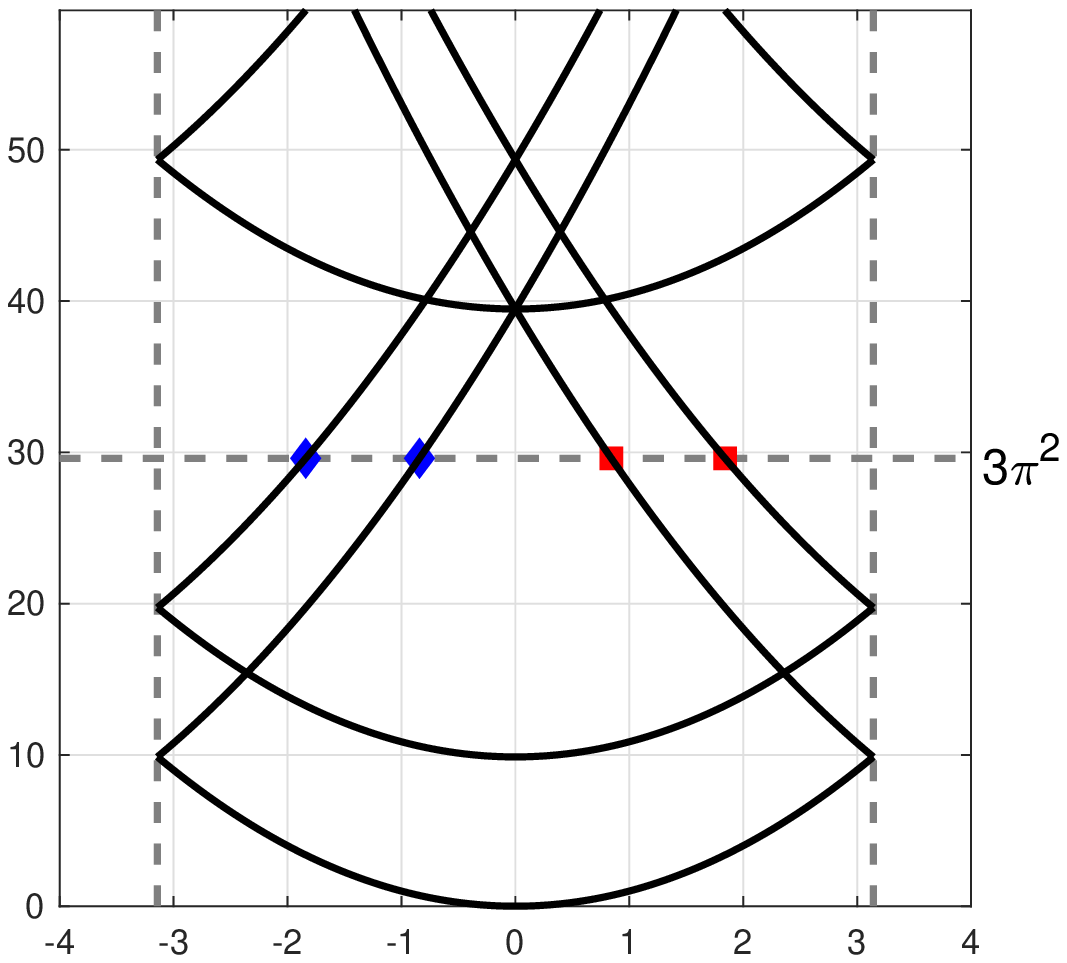} 
& \includegraphics[width=0.4\textwidth,height=0.36\textwidth]{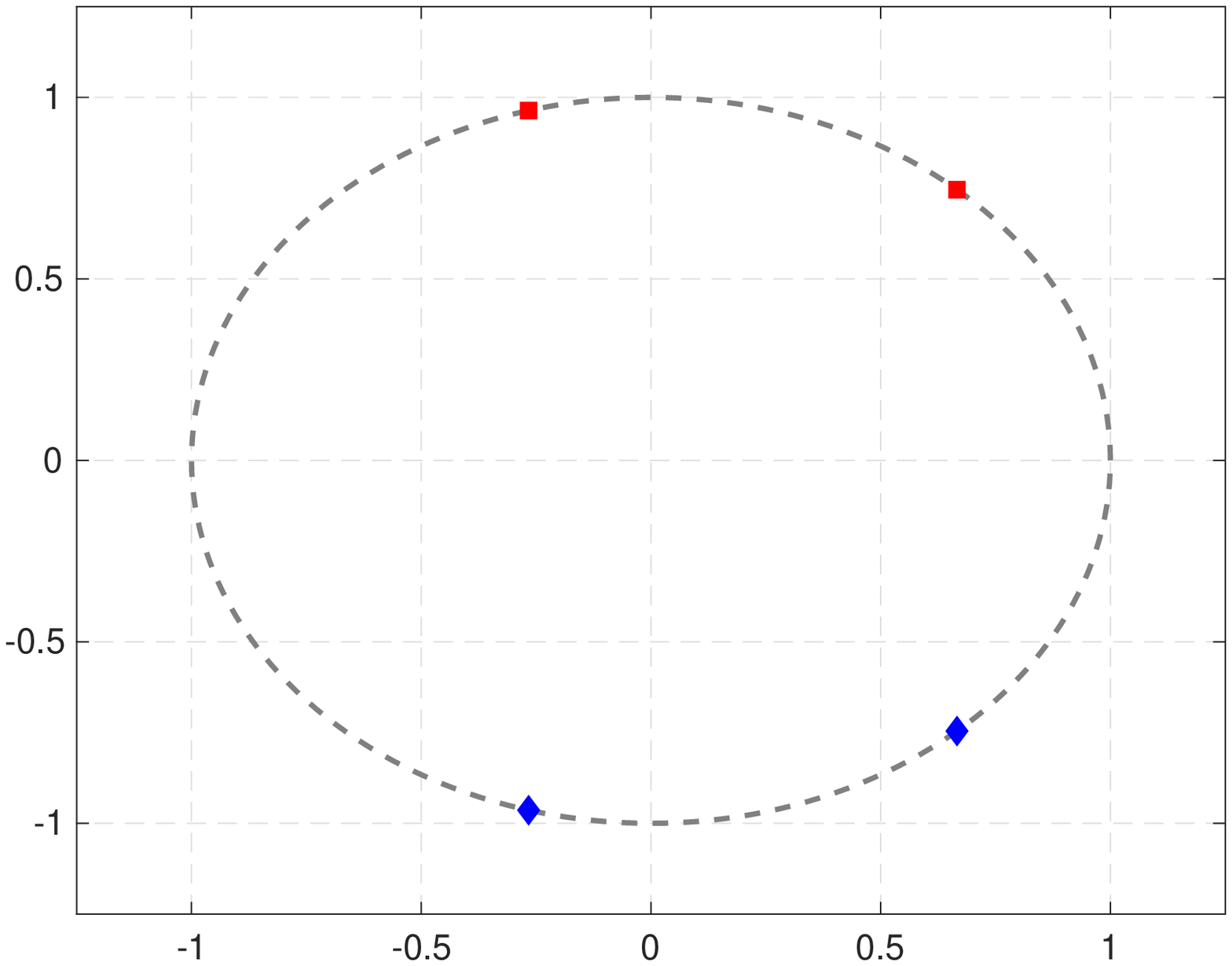}\\[-0cm]
\end{tabular}
\caption{Example for $n=1$ and $k^2=3\pi^2$. $\mathbb{UF}$ in $\alpha$-space and $z$-space. Red squares \high{denote} the points in $P_-$ ($\S_-^0$), while blue diamonds \high{denote} the points in $P_+$ ($\S_+^0$).}
\label{fig:z_alpha}
\end{figure}

 We can also define the subsets of $\mathbb{F}\setminus\mathbb{UF}$ as follows:
\begin{equation*}
 RS:=\left\{z\in\mathbb{F}:\,|z|<1\right\};\quad LS:=\left\{z\in\mathbb{F}:\,|z|>1\right\}.
\end{equation*}
The Bloch wave solution \high{corresponding} to $z\in RS$ is evanescent while the one \high{corresponding} to $z\in LS$ is \high{exponentially increasing}.
\high{Finally, let us set}
\begin{equation*}
 S_+:=S_+^0\cup RS,\quad S_-:=S_-^0\cup LS.
\end{equation*}
When Assumption \ref{asp1} is satisfied,  $S_+^0\cup S_-^0=\mathbb{UF}$ and $S_+\cup S_-=\mathbb{F}$. The distributions of $RS$ and $LS$ are described in the following lemma.

\begin{lemma}\label{th:s_pm_dist}
Let $k\in\C$ satisfy $\Re(k)>0$ and $\Im(k)\geq 0$ . 
 There is a $\tau>0$ such that $RS\high{\subset} B(0,\exp(-\tau))$ and $LS\high{\subset}\C\setminus\overline{B(0,\exp(\tau))}$.
\end{lemma}

\begin{proof}
 If $k^2\in \sigma(A)$, from Property b) and c), $\mathbb{F}$ is discrete and $\mathbb{F}\supset\mathbb{UF}\neq\emptyset$ is a finite set . Let $\mathbb{UF}=\{z_1,\dots,z_N\}$. As $\mathbb{F}$ is discrete, for any $n=1,\dots,N$, there is a $\delta_n>0$ such that $B(z_n,\delta_n)\cap\mathbb{F}=\{z_n\}$. From the perturbation theory, for any $z\in\UC\setminus\mathbb{UF}$, there is a $\delta_z>0$ such that $B(z,\delta_z)\cap\mathbb{F}=\emptyset$. As $\UC\subset\left[\cup_{n=1}^N B(z_n,\delta_n)\right]\cup\left[\cup_{z\in\UC\setminus\mathbb{UF}}B(z,\delta_z)\right]$, from Heine-Borel theorem, there is finite number of $z$'s on $\UC\setminus\mathbb{UF}$, denoted by $z_{N+1},\dots,z_M$ such that $\UC\subset\left[\cup_{n=1}^N B(z_n,\delta_n)\right]\cup\left[\cup_{z=n+1}^M B(z,\delta_z)\right]$. Let $\tau>0$ be a small enough number such that
 \begin{equation*}
 B(0,\exp(\tau))\setminus\overline{B(0,\exp(-\tau))}\subset \left[\cup_{n=1}^N B(z_n,\delta_n)\right]\cup\left[\cup_{z=N+1}^M B(z,\delta_z)\right],
 \end{equation*}
then $\left(RS\cup LS\right)\cap T_\tau=\emptyset$, where
\[T_\tau=\left\{z\in\C:\,\exp(-\tau)<|z|\exp(\tau)\right\}.\]
Thus $RS\in B(0,\exp(-\tau))$ and $LS\in\C\setminus\overline{B(0,\exp(\tau))}$. When $k^2\notin\sigma(\A)$, the result is obtained simply from the perturbation theory and Heine-Borel theorem.  The proof is finished.
 
\end{proof}
}

\section{Floquet-Bloch transform and its application}

\subsection{Floquet-Bloch transform}

The Floquet-Bloch transform is a very important tool in the analysis of scattering problems with periodic structures, see \cite{Kirsc2017,Kirsc2017a,Fliss2015}. The Floquet-Bloch transform is an extension of the Fourier series (see \cite{Lechl2016,Kirsc2017,Kirsc2017a,Fliss2015}). We first recall the definition and some important properties of the Floquet-Bloch transform.

\begin{definition}
 For any $\phi\in C_0^\infty\left(\Omega\right)$, let $(\F \phi)(z,x)$ be the Floquet-Bloch transform of $\phi$ defined by
\begin{equation}\label{eq:def_F}
 (\F \phi)(z,x):=\sum_{\high{n}\in\Z}\phi(x_1+n,{x_2})z^{-n},\quad x\in\Omega_0,\,z\in\UC.
\end{equation}
\end{definition}

Some important properties of the Floquet-Bloch transform are \high{given} in the following theorems.

\begin{theorem}\label{th:prop_F}
The operator $\F $ has the following properties when $z$ lies on the unit circle $\UC$ (see \cite{Lechl2016,Kuchm2016}):
\begin{itemize}
\item $\F$ is an isomorphism between $H^s(\Omega)$ and $L^2(\UC;H^s_z(\Omega_0))$ (where $s\in\R$), where 
\begin{equation*}
L^2(\UC;H^s_z(\Omega_0)):=\left\{\phi\in\mathcal{D}'(\UC\times\Omega_0):\,\left[\int_{\UC}\left\|\phi(z,\cdot)\right\|^2_{H^s_z(\Omega_0)}\d z\right]^{1/2}<\infty\right\}.
\end{equation*}
\item $\F \phi$ depends analytically on $z\in \UC$, if and only if $\phi$ decays exponentially \high{at  infinity}.
\end{itemize}
\end{theorem}

The inverse Floquet-Bloch transform is described in the following theorem.

\begin{theorem}\label{th:inv_F}
Given $\psi(z,x)\in L^2\left(\UC; H^s_z(\Omega)\right)$,  the inverse operator $\F$ is given by:
 \begin{equation}\label{eq:def_invF}
  (\F^{-1}\psi)(x_1+n,{x_2})=\frac{1}{2\pi\i}\oint_{\UC}\psi(z,x)z^{n-1}\d z,\quad\forall\,n\in\Z.
 \end{equation}
\end{theorem}

Now we have enough information \high{about} the Floquet-Bloch transform, the next step is to apply it to the analysis of scattering problems \high{\eqref{eq:wg1}-\eqref{eq:wg2}.}

\subsection{Application of the Floquet-Bloch transform: $\epsilon>0$}

In this section, we apply the Floquet-Bloch transform $\F$ to the scattering problem \eqref{eq:wg1}-\eqref{eq:wg2}, when $k^2$ is replaced by $k^2_\epsilon=k^2+\i\epsilon$ for some fixed $\epsilon>0$. From Lax-Milgram theorem, $k^2_\epsilon\notin\sigma(A)$. \high{From the Floquet-Bloch theory, $\mathbb{UF}(k^2+\i\epsilon)=\emptyset$. Note that in this section, the set of (unit) Floquet multipliers depending on $\epsilon$ is written as $\mathbb{F}(k^2+\i\epsilon)$ ($\mathbb{UF}(k^2+\i\epsilon)$). From Sec. 6.2. in \cite{Joly2006}, any element in the set $\mathbb{F}(k^2+\i\epsilon)$  (in \cite{Joly2006}, it is an eigenvalue $\lambda(\epsilon)$ of the operator $\mathcal{R}^+_\epsilon$) depends continuously on $\epsilon>0$.}

As $f$ is compactly supported and $\epsilon>0$, the problem \eqref{eq:wg1}-\eqref{eq:wg2} has a unique solution $u_\epsilon\in H^1(\Omega)$. Moreover, $u_\epsilon$ decays exponentially at the infinity \high{(see \cite{Ehrhardt2009})}, i.e., $u_\epsilon$ satisfies 
\begin{equation*}
 |u_\epsilon(x_1,x_2)|\leq C\exp(-\gamma|x_1|)
\end{equation*}
for some constants $C>0$ \high{that does not depend on $x_1$} and $\gamma>0$ \high{(see Definition 4 in \cite{Fliss2015})}. We define the Floquet-Bloch transform $w_\epsilon(z,x):=(\F u_\epsilon)(z,x)$, then the transformed field is well-defined in $L^2({\UC};H^1_z(\Omega_0))$. Moreover, it depends analytically on $z\in\UC$. It is also easy to check that for any fixed $z\in \UC$, $w_\epsilon(z,\cdot)\in H_z^1(\Omega_0)$ \high{is a weak solution of}
\begin{eqnarray}\label{eq:wg_z1}
  \Delta w_\epsilon(z,\cdot)+k^2_\epsilon q w_\epsilon(z,\cdot)&=&f\quad\text{ in }\Omega_0;\\\label{eq:wg_z2}
  \frac{\partial w_\epsilon(z,\cdot)}{\partial \nu}&=&0\quad\text{ on }\partial{\Omega_0}.
\end{eqnarray}
As ${\rm supp}(f)\subset\Omega_0$, $(\F f )(z,x)=f(x)$ for any $z\in \UC$.

From Theorem \ref{th:inv_F}, the original field $u_\epsilon$ has the following representation:
\begin{equation}\label{eq:solution_eps}
 u_\epsilon(x_1+n,x_2)= (\F^{-1} w_\epsilon)(x_1+n,x_2)=\frac{1}{2\pi i}\oint_{\UC} w_\epsilon(z,x)z^{n-1}\d z.
\end{equation}
\high{Thus the solution $u_\epsilon$ of \eqref{eq:wg1}-\eqref{eq:wg2} is obtained from the contour integration of its Floquet-Bloch transformed field $w_\epsilon(z,\cdot)$, which is a family of solutions of cell problems \eqref{eq:wg_z1}-\eqref{eq:wg_z2}. On the other hand, from Theorem \ref{th:solution_z_uc},  when $z\in\C\setminus\mathbb{F}\high{(k^2+\i\epsilon)}$, the problem \eqref{eq:wg_z1}-\eqref{eq:wg_z2} is uniquely solvable in $H_z^1(\Omega_0)$. Thus $w_\epsilon(z,\cdot)$ is extended analytically to $z\in\C\setminus(\mathbb{F}\high{(k^2+\i\epsilon)}\cup\{0\})$ and meromorphically to $z\in\C\setminus\{0\}$. Then we can completely forget the Floquet-Bloch transform process and only focus on the relationship \eqref{eq:solution_eps} between two independent problems, i.e., the waveguide problem \eqref{eq:wg1}-\eqref{eq:wg2} and the cell problem \eqref{eq:wg_z1}-\eqref{eq:wg_z2}. The result is concluded in the following theorem.}\\

\begin{theorem}
 For any $k^2+\i\epsilon$ with $\epsilon>0$, the Floquet-Bloch transform $w_\epsilon(z,x):=(\F u_\epsilon)(z,x)$ \high{can} be extended to an analytic function in $\C\setminus(\mathbb{F}\high{(k^2+\i\epsilon)}\cup\{0\})$ by the solution $w_\epsilon(z,x)$ of \high{\eqref{eq:wg_z1}-\eqref{eq:wg_z2}}. Moreover, the function is meromorphic in $\C\setminus\{0\}$ where its poles are exactly all the elements $\mathbb{F}$.
\end{theorem}

\section{Limiting absorption principle: $\epsilon=0$}

In this section, we consider the scattering problem \eqref{eq:wg1}-\eqref{eq:wg2} with  positive wavenumber $k$ \high{satisfies Assumption \ref{asp1}} and the limiting absorption principle is adopted. We begin with the properties of the damped Helmholtz equation \eqref{eq:wg_z1}.

\subsection{Distribution of poles of the damped Helmholtz equations}

\high{Recall that both $S_+^0$ and $S_-^0$ are finite with the same number of elements and both $RS$ and $LS$ are countable. From Assumption \ref{asp1}, $S_+^0\cap S_-^0=\emptyset$, they are  written as}
\begin{eqnarray*}
&S_+^0=\Big\{z_1^+,\dots,z_N^+\Big\};\quad RS=\Big\{z_{N+1}^+,z_{N+2}^+,\dots\Big\};\\
&S_-^0=\Big\{z_1^-,\dots,z_N^-\Big\};\quad LS=\Big\{z_{N+1}^-,z_{N+2}^-,\dots\Big\};
\end{eqnarray*}
where $|z_j^+|=|z_j^-|=1$ for $j=1,2,\dots,N$ and $|z_j^+|<1<|z_j^-|$ for $j=N+1,N+2,\dots$.
\high{Moreover}, $\mathbb{F}=S_+^0\cup S_-^0\cup RS\cup LS$. Note that when $k^2\notin\sigma(A)$, both $S_+^0$ and $S_-^0$ are empty. In this case, we set $N=0$.

From the \high{last section, elements in $\mathbb{F}(k^2+\i\epsilon)$ depend continuously on $\epsilon$. As the poles $z_j^\pm$ ($j\in\N$) are exactly all the elements in $\mathbb{F}(k^2+\i\epsilon)$, they depend continuously on $\epsilon>0$ as well. Thus}  for any $z_j^\pm\in\mathbb{F}$ with $j\in\N$, there is a continuous function $Z_j^\pm(\epsilon)$, such that $\{Z_j^+(\epsilon),Z_j^-(\epsilon):\,j\in\N\}$ is exactly the set of all poles with respect to $k^2+\i\epsilon$. Moreover,  $\lim_{\epsilon\rightarrow 0}Z^\pm_j(\epsilon)=Z^\pm_j(0)=z_j^\pm$. For simplicity, we define the following sets
\begin{eqnarray*}
&S_+^0(\epsilon)=\{Z_1^+(\epsilon),\dots,Z_N^+(\epsilon)\};\quad RS(\epsilon)=\{Z_{N+1}^+(\epsilon),Z_{N+2}^+(\epsilon),\dots\};\\
&S_-^0(\epsilon)=\{Z_1^-(\epsilon),\dots,Z_N^-(\epsilon)\};\quad LS(\epsilon)=\{Z_{N+1}^-(\epsilon),Z_{N+2}^-(\epsilon),\dots\}.
\end{eqnarray*}
For simplicity, let
\begin{equation*}
S_+(\epsilon):=S_+^0(\epsilon)\cup RS(\epsilon),\quad S_-(\epsilon)=S_-^0(\epsilon)\cup LS(\epsilon).
\end{equation*}
The following lemma shows the asymptotic behavior of $\left|Z_j^\pm(\epsilon)\right|$ as $\epsilon\rightarrow 0$ where $j=1,2,\dots,N$, i.e., $Z_j^\pm(\epsilon)\in\S_\pm^0(\epsilon)$, for the proof we refer to Appendix in \cite{Joly2006}.

\begin{lemma}\label{th:curve_z}
For any $j=1,2,\dots,N$, when $\epsilon>0$ is small enough, the functions satisfy $|Z_j^+(\epsilon)|<1$ and $|Z_j^-(\epsilon)|>1$. 
\end{lemma}

%

From this lemma, we know the behavior of the curves $Z_j^\pm(\epsilon)$ for small enough $\epsilon$:
\begin{itemize}
\item for any $z_j^+\in S^0_+$, the points $Z_j^+(\epsilon)$ converges to $z_j^+$ from the inside of the unit circle;
\item for any $z_j^-\in S^0_-$, the points $Z_j^-(\epsilon)$ converges to $z_j^-$ from the outside of the unit circle.
\end{itemize}
To make it clear, we show a visualization of the curves in Figure \ref{fig:curve}. The red rectangles are points in $S_-^0$ and the blue diamonds are points in $S_+^0$. When $\epsilon>0$, we can see that the curve $Z_j^-(\epsilon)$ converges to $z_j^-$ from the exterior of the unit circle, and $Z_j^+(\epsilon)$ converges to $z_j^+$ from the interior of the unit circle.

\begin{figure}[ht]
\centering
\begin{tabular}{c  c}
\includegraphics[width=0.45\textwidth]{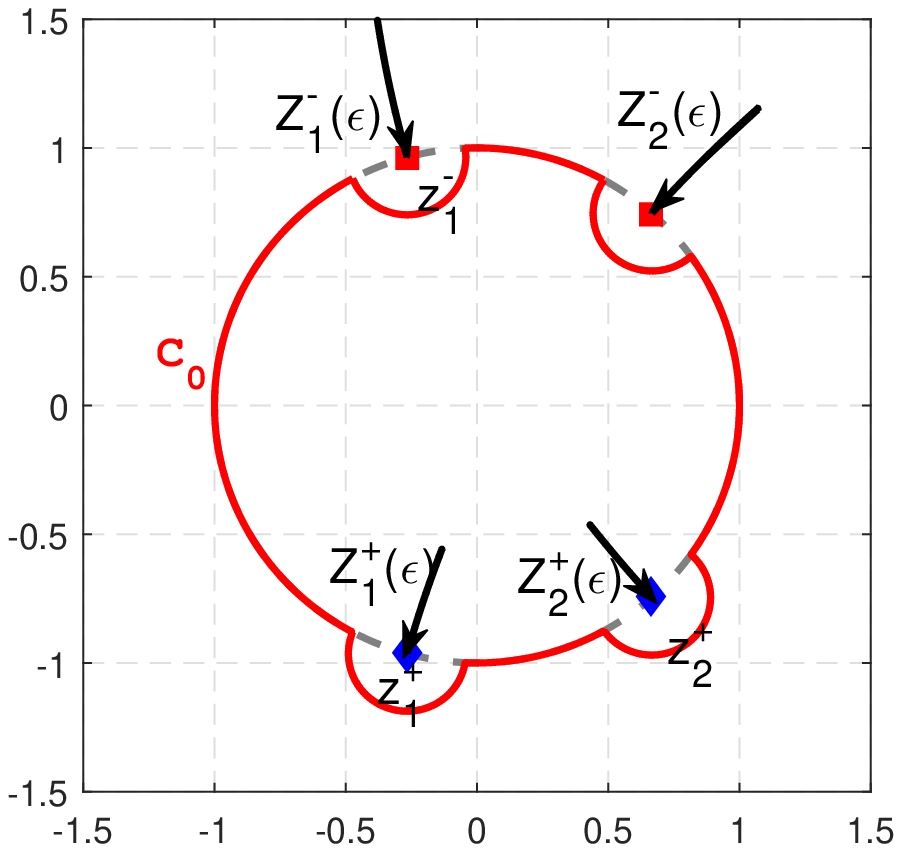} 
& \includegraphics[width=0.45\textwidth]{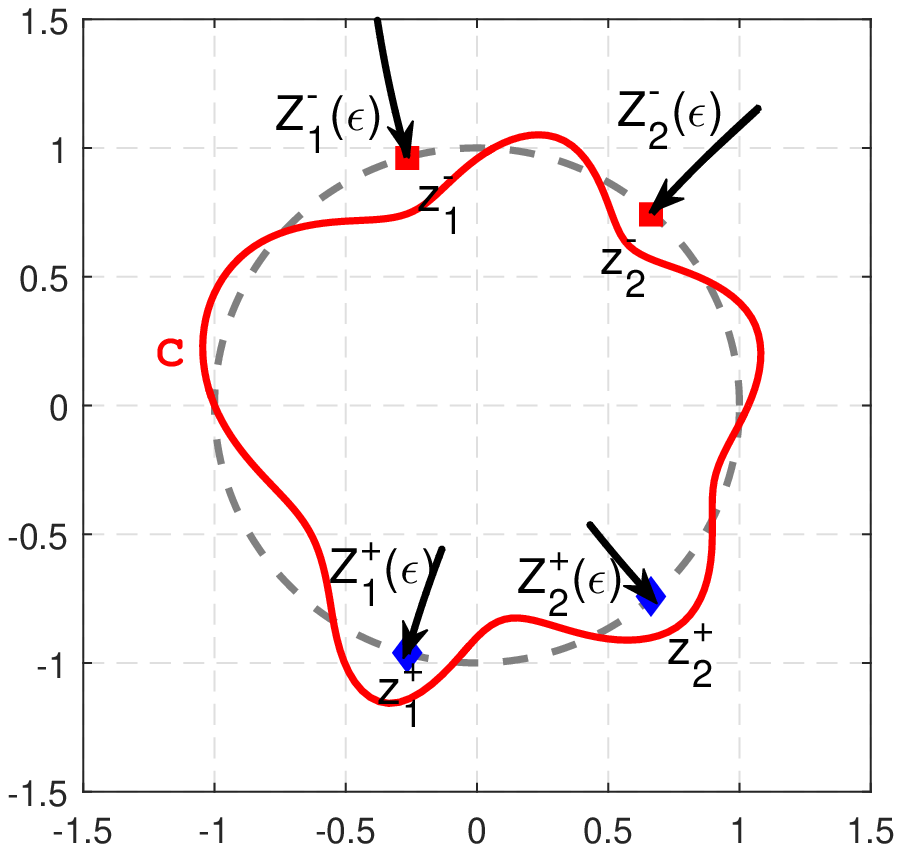}\\[-0cm]
\end{tabular}
\caption{Left: the curve $\mathcal{C}_0$; Right: a choice of $\mathcal{C}$. Black curves: $Z_j^\pm(\epsilon)$.  The red rectangles are points in $S^0_-$ and the blue diamonds are points in $S^0_+$. The arrows show the asymptotic behavior of the poles when $\epsilon\rightarrow 0$.}
\label{fig:curve}
\end{figure}

For the points in $RS(\epsilon)$ or $LS(\epsilon)$, we  also estimate their distributions for small enough $\epsilon>0$ in the following lemma.

\begin{lemma}\label{th:curve_z_inside}
Suppose for some $\tau>0$, $RS\subset B(0,\exp(-\tau))$  and $LS\subset\C\setminus\overline{B(0,\exp(\tau))}$. 
When $\epsilon>0$ is small enough, $Z_j^+(\epsilon)\in B(0,\exp(-\tau_1))$ and $Z_j^-(\epsilon)\in \C\setminus\overline{B(0,\exp(\tau_1))}$ for any $j\geq N+1$, where $\tau_1$ takes any fixed value in $(0,\tau)$.
\end{lemma}

\high{The proof is similar to that of Lemma \ref{th:s_pm_dist}, thus is omitted here.}

From Lemma \ref{th:curve_z} and \ref{th:curve_z_inside}, for $\epsilon>0$ small enough, the sets $S_\pm^0(\epsilon)$ and $RS(\epsilon)$, $RS(\epsilon)$ have the following properties:
\begin{equation}\label{eq:dist_pole}
 S_+(\epsilon)=S_+^0(\epsilon)\cup RS(\epsilon)\subset B(0,1);\quad S_-(\epsilon)=S_-^0(\epsilon)\cup LS(\epsilon)\subset \C\setminus\overline{B(0,1)}.
\end{equation}
Moreover, any point in $S_+^0(\epsilon)$ approaches the unit circle from the interior of $B(0,1)$, and any point in $S_-^0(\epsilon)$ from the exterior of $B(0,1)$. There is a $\tau>0$ such that when $\epsilon$ is small enough, all the points in $RS(\epsilon)$ lie uniformly in the ball $B(0,\exp(-\tau))$, while the points in $LS(\epsilon)$ lie uniformly in the exterior of the ball $B(0,\exp(\tau))$.

\subsection{Limiting absorption process}

In this subsection, we consider the limit of $u_\epsilon$ when $\epsilon\rightarrow 0^+$.  From Lemma \ref{th:curve_z_inside}, we can always find a $\tau>0$ such that the following conditions are satisfied:
\begin{itemize}
 \item No poles of $(I-\K_z)^{-1}$ lies on $\partial B(0,\exp(-\tau))$;
 \item for $\epsilon>0$ small enough, $Z_j^+(\epsilon)\subset B(0,\exp(-\tau))$ for any $j\geq N+1$;
 \item for $\epsilon>0$ small enough, $Z_j^+(\epsilon)\subset B(0,1)\setminus\overline{B(0,\exp(-\tau))}$ for any $j=1,2,\dots,N$.
\end{itemize}

As $w_\epsilon$ depends analytically on $z$ in $\C$ except for $\{Z_j^\pm(\epsilon)\}\cup\{0\}$, from \eqref{eq:solution_eps} and Cauchy integral theorem, for $n\geq 1$,
\begin{equation*}
 u_\epsilon(x_1+n,x_2)=\frac{1}{2\pi\i}\oint_{|z|=\exp(-\tau)}w_\epsilon(z,x)z^{n-1}\d z+\sum_{j=1}^N\left[\frac{1}{2\pi\i}\oint_{\left|z-Z_j^+(\epsilon)\right|=\tau_j}w_\epsilon(z,x)z^{n-1}\d z\right].
\end{equation*}
\high{As $\mathbb{F}(k^2+\i\epsilon)$ is a discrete set,   for any small enough $\epsilon>0$, $\tau_j=\tau_j(\epsilon)>0$ is small enough such that for $j=1,2,\dots,N$, $B(Z_j^+(\epsilon),\tau_j)\subset B(0,1)\setminus\overline{B(0,\exp(-\tau))}$ and the intersection of every two balls is empty. For the integral curve we refer to Figure \ref{fig:cc}.} Note that we only discuss the case when $n\geq 1$, as everything is similar for $n\leq -1$. 
\begin{figure}[ht]
\centering
\includegraphics[width=0.45\textwidth]{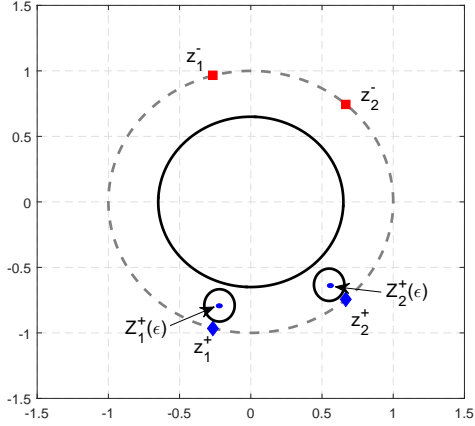} 
\caption{\high{Integral curve: the large black solid circle is $|z|=\exp(-\tau)$, the small black solid circles are $\left|z-Z_j^+(\epsilon)\right|=\tau_j$.}}
\label{fig:cc}
\end{figure}

Define the following functions 
\begin{eqnarray*}
 && u_\epsilon^0(x_1+n,x_2):=\frac{1}{2\pi\i}\oint_{|z|=\exp(-\tau)}w_\epsilon(z,x)z^{n-1}\d z;\\
 && u_\epsilon^j(x_1+n,x_2):=\frac{1}{2\pi\i}\oint_{|z-Z_j^+(\epsilon)|=\tau_j}w_\epsilon(z,x)z^{n-1}\d z,\quad j=1,2,\dots,N.
\end{eqnarray*}

First consider the function $u_\epsilon^0$. As no poles lie on $\partial B(0,\exp(-\tau))$, from the perturbation theory and Heine-Borel theorem, $w_\epsilon(z,x)$ converges to $w(z,x)$ uniformly with respect to $z$, when $\epsilon\rightarrow 0$. Define
\begin{equation}\label{eq:u_0}
 u_0(x_1+n,x_2):=\frac{1}{2\pi\i}\oint_{|z|=\exp(-\tau)}w(z,x)z^{n-1}\d z,
\end{equation}
Then it is easy to prove that
\begin{equation*}
 u_\epsilon^0(x_1+n,x_2)\rightarrow u_0(x_1+n,x_2)\quad\text{ in }H^1(\Omega_0)
\end{equation*}
as $\epsilon\rightarrow 0^+$, for any fixed $n\in\N$.

For the second term, we use the spectrum decomposition of the operator $\A_z$. As was described in Section 3, the differential operator $\A_z$ has countable number of eigenvalues $\{\mu_j(\alpha):\,j\in\N\}$ and corresponding eigenfunctions $\{\psi_j(\alpha,\cdot):\,j\in\N\}$, where $z=\exp(\i\alpha)$. We rearrange the order of the eigenvalues such that
\[\mu_j^+(\alpha_j^+)=\mu_j^-(\alpha_j^-)=k^2,\quad \text{such that }z_j^\pm=\exp(\i\alpha_j^\pm)\text{ for }j=1,2,\dots,N.\]
The corresponding eigenfunction is denoted by $\psi_j^\pm(\alpha,\,\dot)$. The rest of the eigenvalues and eigenfunctions are simply denoted by $\mu_j(\alpha)$ and $\psi_j(\alpha,\cdot)$ for $j=N+1,N+2,\dots$. 
From the definition of $S_+^0$ and $S_-^0$,
\[(\mu_j^+)'(\alpha_j^+)>0,\quad (\mu_j^-)'(\alpha_j^-)<0.\]
Recall the resolvent of the operator $k^2_\epsilon I-\A_z$, then the solution $w_\epsilon(z,x)$ has the following form:
\begin{equation*}
\begin{aligned}
 w_\epsilon(z,x)&=\sum_{j=1}^N\frac{\left<q^{-1}z^{n-1}f,\psi_j^+(-\i\log(z),\cdot)\right>}{\mu_j^+(-\i\log(z))-k^2-\i\epsilon}\psi_j^+(-\i\log(z),x)\\
 &+\sum_{j=1}^N\frac{\left<q^{-1}z^{n-1}f,\psi_j^-(-\i\log(z),\cdot)\right>}{\mu_j^-(-\i\log(z))-k^2-\i\epsilon}\psi_j^-(-\i\log(z),x))\\
 &+\sum_{j=N+1}^\infty\frac{\left<q^{-1}z^{n-1}f,\psi_j(-\i\log(z),\cdot)\right>}{\mu_j(-\i\log(z))-k^2-\i\epsilon}\psi_j(-\i\log(z),x).
 \end{aligned}
\end{equation*}
As for $j=N+1,\dots$, the function $\mu_j(\alpha)$ with $\alpha\in(-\pi,\pi]$ is far away from $k^2$, the third term has no poles in the neighborhood of $\UC$ when $\epsilon>0$ is small enough, thus we only need to consider the first two terms. For any $j=1,2,\dots,N$, $\mu_j^\pm(\alpha)=k^2$ has the solution $\alpha_j^\pm\in (-\pi,\pi]$.  
From Lemma \ref{th:curve_z}, when $\epsilon>0$ is small enough,  the solution of $\mu_j^\pm(-\i\log(z))=k^2+\i\epsilon$ is $Z_j^\pm(\epsilon)$ with $|Z_j^+(\epsilon)|<1$ and $|Z_j^-(\epsilon)|>1$. Thus the second term has no poles in $B(0,1)$, we only need to consider the first term. As the derivative $(\mu^+_j)'(\alpha(0))>0$, for small enough $\epsilon>0$, \high{$(\mu^+_j)'(\alpha(\epsilon))\neq 0$}. This implies that $\frac{\d}{\d z}\mu_j(-\i\log(z))\Big|_{z=Z_j^+(\epsilon)}\neq 0$, so $Z_j^+(\epsilon)$ is a pole of order one. From Residue theorem,
\begin{equation*}
 u_\epsilon^j(x_1+n,x_2)=\frac{\i \left(Z_j^+(\epsilon)\right)^n\left<q^{-1}f,\psi_j^+(-\i\log(Z_j^+(\epsilon)),\cdot)\right>}{(\mu_j^+)'(-\i \log(Z_j^+(\epsilon)))}\psi_j^+(-\i\log(Z_j^+(\epsilon)),x).
\end{equation*}
As $\psi_j^\epsilon$ and $\mu_j^\epsilon$ converges when $\epsilon\rightarrow 0^+$, let
\begin{equation*}
 u_j(x_1+n,x_2)=\frac{\i \left(z_j^+\right)^n\left<q^{-1}f,\psi_j^+(-\i\log(z_j^+),\cdot)\right>}{(\mu_j^+)'(-\i \log(z_j^+))}\psi_j^+(-\i\log(z_j^+),x),
\end{equation*}
then $u_\epsilon^j\rightarrow u_j$ as $\epsilon\rightarrow 0^+$.

From the arguments above, the LAP solution has the following representation \high{for $n\geq 1$:}
\begin{equation}\label{eq:LAP}
\begin{aligned}
 u(x_1+n,x_2)&=\frac{1}{2\pi\i}\oint_{|z|=\exp(-\tau)}w(z,x)z^{n-1}\d z\\
 &+\sum_{j=1}^N\left[\frac{\i \left(z_j^+\right)^n\left<q^{-1}f,\psi_j^+(-\i\log(z_j^+),\cdot)\right>}{(\mu_j^+)'(-\i \log(z_j^+))}\psi_j^+(-\i\log(z_j^+),x)\right].
 \end{aligned}
\end{equation}
Thus the LAP solution in $\cup_{j=1}^\infty\Omega_j$ is composed of finite number of Bloch wave solutions that are propagating to the right and a contour integral defined as \eqref{eq:u_0}. The solution in $\cup_{j=-\infty}^{-1}\Omega_j$ can be obtained in the same way. \high{For $n\leq -1$:
\begin{equation}\label{eq:LAP_neg}
\begin{aligned}
 u(x_1+n,x_2)&=\frac{1}{2\pi\i}\oint_{|z|=\exp(\tau)}w(z,x)z^{n-1}\d z\\
 &+\sum_{j=1}^N\left[\frac{\i \left(z_j^-\right)^n\left<q^{-1}f,\psi_j^-(-\i\log(z_j^-),\cdot)\right>}{(\mu_j^-)'(-\i \log(z_j^-))}\psi_j^-(-\i\log(z_j^-),x)\right].
 \end{aligned}
\end{equation}
}

The representation of LAP solutions and the stability result are concluded in the following theorem.
\begin{theorem}\label{th:stab}
 For any $f\in L^2(\Omega_0)$, there is a unique LAP solution $u$ with the form \eqref{eq:LAP} \high{satisfying}
 \begin{equation*}
  \|u\|_{H^1(\Omega_n)}\leq \high{C} \|f\|_{L^2(\Omega_0)},
 \end{equation*}
 where \high{$C$ is a constant that does not depend on $n$}.

\end{theorem}

\high{
\begin{proof}
For $n\geq 1$, the result comes from the fact that the norm of $w(z,\cdot)$ is uniformly bounded with respect to $z$ on the circle $|z|=\exp(-\tau)$, and $|z_j^+|=1$. For $n\leq -1$, the proof is similar thus is omitted. 
\end{proof}
}

\begin{remark}\label{rem}
 Actually, the condition $f\in L^2(\Omega_0)$ could be replaced by $f\in \high{\left(H^1_z(\Omega_0)\right)^*}$. From similar arguments, we can still prove $u\in H^1_{loc}(\Omega)$. The only difference is that $\Delta u\in L^2_{loc}(\Omega)$ no longer holds.
\end{remark}

\section{The space  of LAP solutions}

In this section, we introduce some properties of the following set:
  \begin{equation*}
  \mathcal{U}:=\left\{{u}(f)\Big|_{\Omega_0}:\,{u}(f)\in H^1_{\loc}(\Omega)\text{ is the LAP solution of \eqref{eq:wg1}-\eqref{eq:wg2} with source }f\in L^2(\Omega_0)\right\}.
 \end{equation*}
\high{ For simplicity, we define an auxiliary set:
  \begin{equation*}
  \widetilde{\mathcal{U}}:=\left\{{u}(f)\Big|_{\Omega_0}:\,{u}(f)\in H^1_{\loc}(\Omega)\text{ is the LAP solution of \eqref{eq:wg1}-\eqref{eq:wg2} with source }f\in \high{\left(H^1_z(\Omega_0)\right)^*}\right\}.
 \end{equation*}

  As was explained in Remark \ref{rem}, the result in Theorem \ref{th:stab} is easily extended to cases where $f\in \left(H^1_z(\Omega_0)\right)^*$. Thus $\mathcal{U}\subset\widetilde{\mathcal{U}}\subset H^1(\Omega_0)$. 
For a better explanation, we introduce the following Hilbert triple (see Section 5.2, \cite{Brezis2011}):
\[
H^1(\Omega_0)\subset L^2(\Omega_0)\subset\left(H^1(\Omega_0)\right)^*.
\] 
As $H^1(\Omega_0)$ is dense in $L^2(\Omega_0)$ and $H^1(\Omega_0)$ is reflexive, $L^2(\Omega_0)$ is also dense in $\left(H^1(\Omega_0)\right)^*$.
 }
 
 To guarantee that the LAP holds, we still assume that $k^2$ satisfies Assumption \ref{asp1}. 
 Thus $\mathcal{U}$ is a linear subspace of $H^1(\Omega_0)$. 
To study the properties of $\mathcal{U}$, we introduce the following integrals first. For any $f,\,g\in \high{\left(H^1(\Omega_0)\right)^*}$, define the following functionals:
 \begin{eqnarray*}
  &&I_1(f,g):=\int_{\Gamma_1}\left[\frac{\partial {u}(g)}{\partial\, x_1}{u}(f)-{u}(g)\frac{\partial {u}(f)}{\partial\, x_1}\right]\d s;\\
  &&I_0(f,g):=\int_{\Gamma_0}\left[\frac{\partial {u}(g)}{\partial\, x_1}{u}(f)-{u}(g)\frac{\partial {u}(f)}{\partial\, x_1}\right]\d s,
 \end{eqnarray*} 
 where $u(f)$ and $u(g)$ are LAP solutions with sources $f$ and $g$.  We can also define the operators that depend on $\epsilon>0$:
  \begin{eqnarray*}
  &&I_1^\epsilon(f,g):=\int_{\Gamma_1}\left[\frac{\partial {u}_\epsilon(g)}{\partial x_1}{u}_\epsilon(f)-{u}_\epsilon(g)\frac{\partial {u}_\epsilon(f)}{\partial x_1}\right]\d s;\\
  &&I_0^\epsilon(f,g):=\int_{\Gamma_0}\left[\frac{\partial {u}_\epsilon(g)}{\partial x_1}{u}_\epsilon(f)-{u}_\epsilon(g)\frac{\partial {u}_\epsilon(f)}{\partial x_1}\right]\d s;
 \end{eqnarray*} 
where $u_\epsilon(f)$ and $u_\epsilon(g)$ are unique solutions of \eqref{eq:wg1}-\eqref{eq:wg2} with $k^2$ replaced by $k^2_\epsilon=k^2+\i\epsilon$ ($\epsilon>0$).

\begin{lemma}\label{th:bdint}
For any $f,\,g\in  \high{\left(H^1(\Omega_0)\right)^*}$,  $I_0(f,g)=I_1(f,g)=0$.

\end{lemma}

\high{Note that a similar result has been shown in Theorem 3, \cite{Fliss2015} for a more general situation.}

\begin{proof}
 We first consider $I_1^\epsilon(f,g)$. 
As both solutions ${u}_\epsilon(f)$ and ${u}_\epsilon(g)$ satisfy the homogeneous Helmholtz equation $\Delta u_\epsilon+(k^2+\i\epsilon)qu_\epsilon=0$ in $\Omega_1$, use the Green's formula,
\begin{equation*}
\begin{aligned}
0&=\int_{\Omega_1}\Big({u}_\epsilon(f)\Delta {u}_\epsilon(g) - {u}_\epsilon(g)\Delta {u}_\epsilon(f)\Big)\d x\\
&=\left(\int_{\Gamma_2}-\int_{\Gamma_1}\right)\left[\frac{\partial {u}_\epsilon(g)}{\partial x_1}{u}_\epsilon(f)-{u}_\epsilon(g)\frac{\partial {u}_\epsilon(f)}{\partial x_1}\right]\d s\\
&=\int_{\Gamma_2}\left[\frac{\partial {u}_\epsilon(g)}{\partial x_1}{u}_\epsilon(f)-{u}_\epsilon(g)\frac{\partial {u}_\epsilon(f)}{\partial x_1}\right]\d s-I_1^\epsilon(f,g).
\end{aligned}
\end{equation*}
Thus $I_1^\epsilon(f,g)=\int_{\Gamma_2}\left[\frac{\partial {u}_\epsilon(g)}{\partial x_1}{u}_\epsilon(f)-{u}_\epsilon(g)\frac{\partial {u}_\epsilon(f)}{\partial x_1}\right]\d s$. 
Use the Green's formula repeatedly in $\Omega_2,\dots,\Omega_{n-1}$, then for $n\geq 3$,
\begin{equation*}
I_1^\epsilon(f,g)=\int_{\Gamma_n}\left[\frac{\partial {u}_\epsilon(g)}{\partial x_1}{u}_\epsilon(f)-{u}_\epsilon(g)\frac{\partial {u}_\epsilon(f)}{\partial x_1}\right]\d s
\end{equation*}
for any integer $n\geq 2$. With the result of Theorem 3.10, \cite{Cakon2006}, there is a $C>0$ independent of $n$ such that
\begin{equation*}
 \left|I_1^\epsilon(f,g)\right|\leq C\|{u}_\epsilon(f)\|_{H^1(\Omega_n)}\|{u}_\epsilon(g)\|_{H^1(\Omega_n)}.
\end{equation*}
From the exponential decay of both functions ${u}_\epsilon(f)$ and ${u}_\epsilon(g)$, $I_1^\epsilon(f,g)=0$. As ${u}(f)$ and ${u}(g)$ are LAP solutions, $\lim_{\epsilon\rightarrow 0^+}u_\epsilon(f)=u(f)$ and $\lim_{\epsilon\rightarrow 0^+}u_\epsilon(g)=u(g)$ in $H^1_{loc}(\Omega)$, thus
\begin{equation*}
I_1(f,g)=\lim_{\epsilon\rightarrow 0}\int_{\Gamma_1}\left[\frac{\partial {u}_\epsilon(g)}{\partial x_1}{u}_\epsilon(f)-{u}_\epsilon(g)\frac{\partial {u}_\epsilon(f)}{\partial x_1}\right]\d s=0.                                                                                                                                    
\end{equation*}
\high{We can prove that $I_0(f,g)=I_1(f,g)=0$ with the same technique. }
The proof is finished.
\end{proof}

Then we are prepared to prove the density of the space $\mathcal{U}$  in the following lemma.

\high{
\begin{lemma}
 The space $\mathcal{U}$ is dense in $H^1(\Omega_0)$.
\end{lemma}

\begin{proof}
 As $L^2(\Omega_0)$ is dense in $\left(H^1(\Omega_0)\right)^*$, $\overline{\mathcal{U}}=\overline{\widetilde{\mathcal{U}}}$. Thus we only need to prove that $\overline{\widetilde{\mathcal{U}}}=H^1(\Omega_0)$.
 
 We prove by contradiction. 
Suppose the set $\widetilde{\mathcal{U}}$ is not dense in $H^1(\Omega_0)$, i.e., $\overline{\widetilde{\mathcal{U}}}\neq H^1(\Omega_0)$. Then there is  a  non-zero $\phi\in \left(H^1(\Omega_0)\right)^*$ such that
\begin{equation*}
\int_{\Omega_0}{u}(f)\overline{\phi}\d x=0,\quad\forall f\in \left(H^1(\Omega_0)\right)^*.
\end{equation*}
For any $g,\psi\in L^2(\Omega_0)$, let $u(g),{u}(\overline{\psi})$ be the LAP solutions of \eqref{eq:wg1}-\eqref{eq:wg2} with the sources $g$ and $\overline{\psi}$. Use the result in Lemma \ref{th:bdint} and Green's formula,
\begin{equation*}
\begin{aligned}
\int_{\Omega_0}{u}(g)\overline{\psi}\d x&=\int_{\Omega_0}{u}(g)\left[\Delta {u}(\overline{\psi})+k^2q{u}(\overline{\psi})\right]\d x\\
&=\int_{\Omega_0}{u}(\overline{\psi})\left[\Delta {u}(g)+k^2q{u}(g)\right]\d x+\left(\int_{\Gamma_1}-\int_{\Gamma_0}\right)\left[\frac{\partial {u}(\overline{\psi})}{\partial x_1}{u}(g)-{u}(\overline{\psi})\frac{\partial {u}(g)}{\partial x_1}\right]\d s\\
&=\int_{\Omega_0}{u}(\overline{\psi})g\d x+I_1(g,\overline{\psi})-I_0(g,\overline{\psi})=\int_{\Omega_0}{u}\left(\overline{\psi}\right)g\d x.
\end{aligned}
\end{equation*}
 From the density of $L^2(\Omega_0)$ in the space $\left(H^1(\Omega_0)\right)^*$, this equation is easily extended to $g,\psi\in\left(H^1(\Omega_0)\right)^*$. Then $0=\int_{\Omega_0}u(f)\overline{\phi}\d x=\int_{\Omega_0}u(\overline{\phi})f\d x$ for any $f\in\left(H^1(\Omega_0)\right)^*$, thus implies  $u(\overline{\phi})=0$ in $H^1(\Omega_0)$. Then $\phi=0$ in $\left(H^{1}(\Omega_0)\right)^*$, which contradicts with the assumption that $\phi\neq 0$. So $\widetilde{\mathcal{U}}$ is dense in $H^1(\Omega_0)$, which implies that $\overline{\mathcal{U}}=H^1(\Omega_0)$. The proof is finished.

\end{proof}
}

In the following theorem, we prove that $\mathcal{U}$ is not only dense, but also equal to $H^1(\Omega_0)$.

\begin{theorem}\label{th:U}
 The space $\mathcal{U}$ is closed, thus $\mathcal{U}=H^1(\Omega_0)$.
\end{theorem}

\begin{proof}

Suppose there is a sequence $\{f_n\}_{n=1}^\infty\subset L^2(\Omega_0)$ such that ${u}(f_n)$
is a Cauchy sequence in $H^1(\Omega_0)$. To prove that $\mathcal{U}$ is closed, we have to find out an $f_0\in L^2(\Omega_0)$ such that $\lim_{n\rightarrow\infty}{u}(f_n)={u}(f_0)$ in $H^1(\Omega_0)$. 
From the Riesz representation theorem, for any $n\in\N$, there is an $\widetilde{f}_n\in H^1(\Omega_0)$ such that
\begin{equation*}
\int_{\Omega_0}f_n\overline{\phi}\d x=\left<\widetilde{f}_n,\phi\right>,\quad\text{ for all }\phi\in H^1(\Omega_0),
\end{equation*}  
where $\left<\cdot,\cdot\right>$ is the inner product in the space  $H^1(\Omega_0)$.

As ${u}(f)$ depends linearly on $f$, for any $m,n\in\N$, ${u}(f_m-f_n)={u}(f_m)-{u}(f_n)$. Then ${u}(f_m-f_n)$ is the LAP solution of \eqref{eq:wg1}-\eqref{eq:wg2} with source $f_m-f_n$, i.e.,
\begin{equation*}
\Delta {u}(f_m-f_n)+k^2 q {u}(f_m-f_n)=f_m-f_n\quad\text{ in }\Omega_0;\quad \frac{\partial {u}(f_m-f_n)}{\partial\nu}=0\quad\text{ on }\partial\,\Omega_0.
\end{equation*}
Multiply the above equation with $\overline{\phi}$ where $\phi$  is any function in  $C^\infty_0(\Omega_0)$, and apply the Green's formula, then
\begin{equation*}
\int_{\Omega_0}(f_m-f_n)\overline{\phi}\d x=\int_{\Omega_0}\left[k^2 q {u}(f_m-f_n)\overline{\phi}-\nabla {u}(f_m-f_n)\cdot\nabla\overline{\phi}\right]\d x.
\end{equation*}
Thus there is a constant $C=\max\{k^2\|q\|_\infty,1\}$ such that
\begin{equation*}
\left|\left<\widetilde{f}_m-\widetilde{f}_n,\phi\right>\right|=\left|\int_{\Omega_0}(f_m-f_n)\overline{\phi}\d x\right|\leq C\|{u}(f_m-f_n)\|_{H^1(\Omega_0)}\|\phi\|_{H^1(\Omega_0)}.
\end{equation*}
This implies that
\begin{equation*}
\left\|\widetilde{f}_m-\widetilde{f}_n\right\|_{H^{1}(\Omega_0)}\leq C\|{u}(f_m)-{u}(f_n)\|_{H^1(\Omega_0)}.
\end{equation*}
As $\{{u}(f_n)\}_{n=1}^\infty$ is a Cauchy sequence, $\left\{{f}_n\right\}_{n=1}^\infty$ is also a Cauchy sequence in $\left(H^1(\Omega_0)\right)^*$. As the space $\left(H^1(\Omega_0)\right)^*$ is closed, there is an $\widetilde{f}_0\in H^1(\Omega_0)$ such that $\widetilde{f}_n\rightarrow \widetilde{f}_0$, $n\rightarrow\infty$. 
From the choice of $\widetilde{f}_m$ and $\widetilde{f}_n$, for any $\phi\in H^1(\Omega_0)$,
\begin{equation*}
\int_{\Omega_0}\left(f_m-f_n\right)\overline{\phi}\d x=\left<\widetilde{f}_m-\widetilde{f}_n,\phi\right>\rightarrow 0,\quad m,n\rightarrow\infty.
\end{equation*}
From the fact that $H^1(\Omega_0)$ is a dense subspace of $L^2(\Omega_0)$, for any $\phi\in L^2(\Omega_0)$, $\int_{\Omega_0}\left(f_m-f_n\right)\overline{\phi}\d x\rightarrow 0$ as $m,\,n\rightarrow\infty$. Thus $\{f_n\}$ is a Cauchy sequence in $L^2(\Omega_0)$. 
From Theorem \ref{th:stab}, ${u}(f_n)\rightarrow {u}(f_0)$ in $H^1(\Omega_0)$. Thus $\mathcal{U}$ is closed, which implies that $\mathcal{U}=H^1(\Omega_0)$.  The proof is finished.
\end{proof}

From Theorem \ref{th:U}, any function in $H^1(\Omega_0)$ is an LAP solution of \eqref{eq:wg1}-\eqref{eq:wg2} with some $f\in L^2(\Omega_0)$. Define the set
\begin{equation*}
\mathcal{V}:=\left\{{u}\big|_{\Gamma\high{_1}}:\,{u}\in\mathcal{U}\right\},
\end{equation*}
then the following corollary is a direct result from Theorem \ref{th:U} and the trace theorem.

\begin{corollary}\label{th:U_co}
 Let $\mathcal{V}$ be defined as above, then $\mathcal{V}=H^{1/2}(\Gamma)\high{=X}$.
\end{corollary}


\section{Spectrum Decomposition of periodic operators}

 In this section, we get back to the operator $\mathcal{R}$ defined at the beginning. As was discussed, it is an operator from $X=H^{1/2}(\Gamma)$ to itself. Furthermore,
\begin{equation*}
u\big|_{\Gamma_{j+1}}=\mathcal{R} u\big|_{\Gamma_j},\quad \forall\,j\geq 1.
\end{equation*}
In this section, we consider the spectrum decomposition of the operator $\mathcal{R}$. It  has been conjectured and already been applied to numerical simulations in \cite{Joly2006} that the generalized eigenfunctions of the operator $\mathcal{R}$ form a complete set in $L^2(\Gamma_1)$, when $\epsilon>0$. \high{The proof was given in \cite{Hohag2013} when the scattering problem \eqref{eq:wg1}-\eqref{eq:wg2} is assumed to be uniquely solvalbe.      In this section, a different method is utilized for much more general cases.}


\subsection{Generalized residue theorem and application}

From the representation \eqref{eq:LAP}, as $w(z,\cdot)$ is meromorphic with respect to $z$ in the ball $B(0,\exp(-\tau))$ except for $\{0\}$, a natural idea is to apply the Residue theorem. From the definition of $\tau$, all the poles in $B(0,\exp(-\tau))$ are exactly all the elements of the set $RS$. However, as $RS$ has infinite number of elements and $\{0\}$ is the only accumulation point, classic Residue theorem is not available. In this section, we prove that the residue theorem \high{can} be extended to the domain with infinite number of poles when certain conditions are satisfied. We call the extension the {\em Generalized Residue Theorem}.

\begin{theorem}[Generalized Residue Theorem]\label{th:generalized_residue}
Suppose $\D$ is a simply connected open domain in $\C$. Suppose $f(z)$ is a meromorphic function in $\D\setminus\{z_0\}$, where $z_0$ is a point \high{lying} in $\D$. Moreover, $z_0$ may or may not be a pole of $f$ but it is the only  accumulation point of poles $\{z_j:\,j\in\N\}\subset D$. If there is a series $\{r_n\}$ \high{satisfying} (see Figure \ref{circle})
\begin{equation*}
r_1>r_2>\cdots>r_n>\cdots> 0,\quad\lim_{n\rightarrow\infty}r_n=0
\end{equation*}
such that $f(z)$ does not have poles when $|z-z_0|=r_n$ and 
\begin{equation*}
\frac{1}{2\pi\i}\oint_{|z-z_0|=r_{\high{n}}}f(z)\d z\rightarrow 0,\quad n\rightarrow\infty,
\end{equation*}
then  the following integral formula holds:
\begin{equation}\label{eq:gen_residue}
\frac{1}{2\pi\i}\oint_{\partial \D_0}f(z)\d z=\sum_{j\in\N:\,z_j\in\D_0}{\rm Res}\left(f(z),z=z_j\right),
\end{equation}
where $\D_0\subset\D$ is an open subdomain of $\D$ that satisfies $\overline{\D_0}\subset\D$. 
\end{theorem}

\begin{proof}When $z_0\notin\D_0$, as there are only finite number of poles in $\D\setminus\D_0$, the result is obtained directly from classic residue theorem. Thus we only consider the case that $z_0\in\D_0$.

As $z_0$ is the only accumulation point of $\{z_j:\,j\in\N\}$,  the number of poles in 
\begin{equation*}
 \D_n:=\{z\in\D_0:\,|z-z_0|>r_n\}
\end{equation*}
 is finite. As $\lim_{n\rightarrow\infty}r_n=0$,
\begin{equation*}
\cup_{n\in\N}\D_n=\D_0\setminus\{z_0\}.
\end{equation*}
As $z_0$ is the only accumulation point of poles,
\begin{equation*}
 \cup_{n\in\N}\Big[\D_n\cap \{z_j:\,j\in\N\}\Big]=\Big\{j\in\N:\,z_j\in\D_0\Big\}.
\end{equation*}
Apply the classic residue theorem to the domain $\D_n$ with boundary $\partial \D_0$ and $\partial B(z_0,r_n)$, then
\begin{equation*}
\frac{1}{2\pi\i}\oint_{\partial \D_n}f(z)\d z=\frac{1}{2\pi\i}\oint_{\partial \D_0}f(z)\d z-\frac{1}{2\pi\i}\oint_{|z-z_0|=r_n}f(z)\d z=\sum_{j\in\N:\,z_j\in \D_n}{\rm Res}(f(z),z=z_j).
\end{equation*}
Thus 
\begin{equation*}
\left|\sum_{j\in\N:\,z_j\in \D_n}{\rm Res}(f(z),z=z_j)\right|=\left|\frac{1}{2\pi\i}\oint_{\partial \D_0}f(z)\d z-\frac{1}{2\pi\i}\oint_{|z-z_0|=r_n}f(z)\d z\right|
\end{equation*}
As the integral $\frac{1}{2\pi\i}\oint_{|z-z_0|=r_n}f(z)\d z\rightarrow 0$ as $n\rightarrow \infty$, the series $\sum_{j\in\N:\,z_j\in \D_n}{\rm Res}(f(z),z=z_j)$ is uniformly bounded as $n\rightarrow\infty$. Thus
\begin{equation*}
\left|\frac{1}{2\pi\i}\oint_{\partial \D_0}f(z)\d z-\sum_{j\in\N:\,z_j\in \D_n}{\rm Res}(f(z),z=z_j)\right|=\left|\frac{1}{2\pi\i}\oint_{|z-z_0|=r_n}f(z)\d z\right|\rightarrow 0
\end{equation*}
as $n\rightarrow \infty$, the equation \eqref{eq:gen_residue} is proved. The proof is finished.

\end{proof}

\high{
We can easily extend the result in Theorem \ref{th:generalized_residue} to functions in $L^2(\Omega_0)$.
\begin{theorem}
\label{th:generalized_residue_l2}
Suppose the settings are the same as in Theorem \ref{th:generalized_residue}, $f(z,x)$ is a function that depends on both $z$ and $x$. For any fixed $z\in \D\setminus\{z_j:\,j\in\N\}$, $f(z,\cdot)\in L^2(\Omega_0)$. $f(z,\cdot)$ depends analytically on $z\in \D\setminus\{z_j:\,j\in\N\}$ and meromorphically on $z\in\D$ with poles at $\{z_j:\,j\in\N\}$, where the analytical dependence is defined in Definition \ref{def} and the meromorphical dependence is defined in the similar way. Suppose $f(z,x)$ does not have any poles on $|z-z_0|=r_n$ and
\[
\frac{1}{2\pi\i}\oint_{|z-z_0|=r_n}f(z,\cdot)\d z\rightarrow 0\,\text{ in }L^2(\Omega_0),\quad n\rightarrow\infty,
\]
then the following integral formula holds:
\begin{equation}
\label{eq:generalized_residue_l2}
\frac{1}{2\pi\i}\oint_{\partial \D_0}f(z,\cdot)\d z=\sum_{j\in\N:\,z_j\in\D_0}{\rm Res}(f(z,\cdot),z=z_j)\quad\text{ in }L^2(\Omega_0).
\end{equation}
This also implies that
\[
\frac{1}{2\pi\i}\oint_{\partial \D_0}f(z,\cdot)\d z=\sum_{j\in\N:\,z_j\in\D_0}{\rm Res}(f(z,\cdot),z=z_j)\quad\text{ almost everywhere in }\Omega_0.
\]
\end{theorem}


}

Now we apply the newly established generalized residue theorem \high{for $L^2$-functions in Theorem \ref{th:generalized_residue_l2}} to the integral in \eqref{eq:LAP} on the closed curve $\partial B(0,\exp(-\tau))$.  
Recall the integral representation of $u_0$, i.e.,
\begin{equation*}
 u_0(x_1+n,x_2)=\frac{1}{2\pi\i}\oint_{|z|=\exp(-\tau)}w(z,x)z^{n-1}\d z,\quad\forall\, n\geq 1.
\end{equation*}
First, we need to know the distribution of poles. From previous sections, the set $RS$ lies in $B(0,\exp(-\tau))$.
As $z_j^+\rightarrow 0$ as $j\rightarrow\infty$, $0$ is the only accumulation point of the poles. From Theorem \ref{th:generalized_residue_l2}, we only need to find a strictly decreasing series $\{r_\ell\}_{\ell=1}^\infty$ \high{satisfying} that $r_\ell\rightarrow 0$, such that the integral
\begin{equation*}
I_n(r_\ell):=\oint_{|z|=r_\ell}w(z,x)z^{n-1}\d z\rightarrow 0\high{\,\text{ in }L^2(\Omega_0) \text{, as }} \ell\rightarrow\infty,
\end{equation*}
\high{for any $n\geq 1$. }
Before the estimation of $I_n(r_\ell)$, we need a classical Minkowski integral inequality.

\begin{lemma}[Theorem 202, \cite{Hardy1988}]\label{th:minkowski}
Suppose $(S_1,\mu_1)$ and $(S_2,\mu_2)$ are two measure spaces and $F:\,S_1\times S_2\rightarrow\R$ is measurable. Then the following inequality holds for any $p\geq 1$
\begin{equation*}
\left[\int_{S_2}\left|\int_{S_1}F(y,z)\d \mu_1(y)\right|^p \d\mu_2(z)\right]^{1/p}\leq \int_{S_1}\left(\int_{S_\high{2}}|F(y,z)|^p\d\mu_2(z)\right)^{1/p}\d\mu_1(y).
\end{equation*}
\end{lemma}

\begin{figure}[ht]
\centering
\includegraphics[width=8cm]{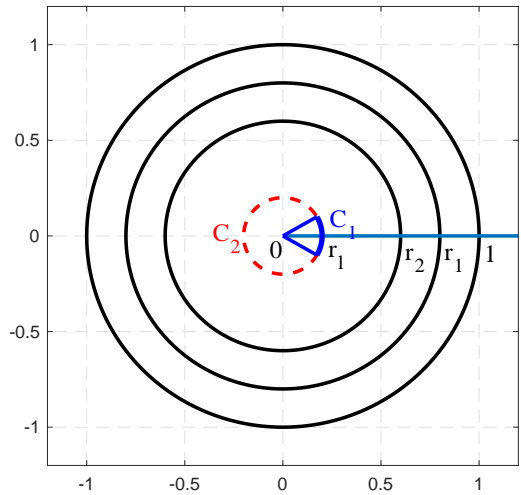}
\caption{Example of the choice of $r_\ell$. For fixed $\ell\in\N$, the blue solid curve is $\mathcal{C}_1^\ell$, the red dotted curve is $\mathcal{C}_2^\ell$}.
\label{circle}
\end{figure}

\begin{lemma}\label{th:int_lm}
There is a series $\{r_\ell\}_{\ell\in\N}$ \high{satisfying} $r_1>r_2>\cdots>r_\ell>\cdots>0$ and $\lim_{n\rightarrow\infty}r_n=0$ such that the integral $I_n(r_\ell)\rightarrow 0$ as $\ell\rightarrow\infty$ \high{in $L^2(\Omega_0)$,} for any $n\geq 1$.
\end{lemma}

\begin{proof}

Let $w(z,\cdot)$ be the solution of the quasi-periodic problem \eqref{eq:quasi1}-\eqref{eq:quasi2}, then  $v_z(x):=\zeta_z^{-1}w(z,x)\in H^1_\p(\Omega_0)$ satisfies
\begin{eqnarray*}
\Delta v_z+2\log(z)\frac{\partial v_z}{\partial x_1}+\log^2(z) v_z&=&g_z\,(:=z^{-x_1}f-k^2 qv_z)\quad \text{ in }\Omega_0;\\
\frac{\partial v_z}{\partial x_2}&=&0\quad\text{ on }\partial \Omega_0.
\end{eqnarray*}

 Let $v_z$ and $g_z$ be written as expansions of the eigenfunctions with homogeneous Neumann boundary conditions, i.e.,
\begin{eqnarray*}
&& v_z(x_1,x_2)=\sum_{j\in\Z}\sum_{m\in\N}v_z^{j,m}\exp(2\i\pi j x_1)\cos(\pi m x_2);\\
&&g_z(x_1,x_2)=\sum_{j\in\Z}\sum_{m\in\N}g_z^{j,m}\exp(2\i\pi j x_1)\cos(\pi m x_2).
\end{eqnarray*}
\high{Using} the Fourier expansions and equation, we get the equation of the coefficients
\begin{equation*}
\left[-4\pi^2 j^2-\pi^2 m^2+2\i\log(z)\pi j+\log^2(z)\right]v_z^{j,m}=g_z^{j,m}.
\end{equation*}
Let $z=r e^{\i\theta}$ when $r>0$ is small enough and $\theta\in[-\pi,\pi]$ (it could be replaced by any interval with length $2\pi$), then $|z|=r$ and $\log(z)=\log(r)+\i\theta$. 
\begin{equation*}
c_{j,m,z}v_z^{j,m}=g_z^{j,m},
\end{equation*}
where the coefficient is defined by
\begin{equation*}
c_{j,m,z}:=\left(-4\pi^2 j^2-\pi^2m^2-2\theta\pi j+\log^2(r)-\theta^2\right)+2\i\left(\pi j+\theta\right)\log(r).
\end{equation*}

1) 
First we consider the case $|\theta|\geq \epsilon>0$ for some small enough $\epsilon>0$.  Then for any $j\in\Z$ and $m\in\N$, the coefficient
\begin{equation*}
|c_{j,m,z}|\geq |\Im(c_{j,m,z})|=2|\log(r)||\pi j+\theta|\geq 2\epsilon|\log(r)|.
\end{equation*}
Thus  $2 \epsilon|\log(r)||v_z^{j,m}|\leq|g_z^{j,m}|$, implies that
\begin{equation*}
2 \epsilon|\log(r)|\|v_z\|_{L^2(\Omega_0)}\leq \|g\|_{L^2(\Omega_0)}\leq \|z^{-x_1}f\|_{L^2(\Omega_0)}+k^2\|q\|_\infty\|v_z\|_{L^2(\Omega_0)}.
\end{equation*}
In this case, when \high{$r$ is small enough, i.e.,} $2\epsilon |\log(r)|>k^2\|q\|_\infty$,
\begin{equation*}
\|v_z\|_{L^2(\Omega_0)}\leq \frac{\left\|z^{-x_1}f\right\|_{L^2(\Omega_0)}}{2\epsilon|\log(r)|-k^2\|q\|_\infty}\leq \frac{\left\|z^{-x_1}\right\|_{L^\infty(\Omega_0)}\|f\|_{L^2(\Omega_0)}}{2\epsilon|\log(r)|-k^2\|q\|_\infty}\leq \frac{r^{-1/2}\|f\|_{L^2(\Omega_0)}}{2\epsilon|\log(r)|-k^2\|q\|_\infty}.
\end{equation*} 
Thus 
\begin{equation}\label{eq:int_est_1}
\|w(z,\cdot)\|_{L^2(\Omega_0)}=\left\|z^{x_1}v_z\right\|_{L^2(\Omega_0)}\leq\left\|z^{x_1}\|_{L^\infty(\Omega_0)}\|v_z\right\|_{L^2(\Omega_0)}\leq \frac{\|f\|_{L^2(\Omega_0)}}{r(2\epsilon|\log(r)|-k^2\|q\|_\infty)}.
\end{equation}

2) 
Second we consider the case $\theta\in(-\epsilon,\epsilon)$. When $j\neq 0$, the coefficient
\begin{equation*}
|c_{j,m,z}|\geq |\Im(c_{j,m,z})|=2|\log(r)||\pi j+\theta|\geq \pi |\log(r)|.
\end{equation*}
When $j=0$, then
\begin{equation*}
|c_{j,m,z}|\geq|\Re(c_{j,m,z})|=\left|\log^2(r)-\pi^2 m^2-\theta^2\right|\geq \left|\log^2(r)-\pi^2 m^2\right|-\epsilon^2.
\end{equation*}
\high{For any $\ell\in\N$, we} choose $r_\ell=\exp\left(-\pi\sqrt{\frac{\ell^2+{(\ell+1)}^2}{2}}\right)$, then we have the following estimations:
\begin{eqnarray*}
&& \pi|\log(r)|\geq \pi^2 \sqrt{\frac{\ell^2+(\ell+1)^2}{2}}\geq \pi^2\ell;\\
&&\left|\log^2(r_\ell)-\pi^2 m^2\right|\geq \left|\log^2(r_\ell)-\pi^2\ell^2\right|=\left|\log^2(r_\ell)-\pi^2{(\ell+1)^2}\right|=\frac{(2\ell+1)\pi^2}{2}. 
\end{eqnarray*}
Thus when $\epsilon$ is small enough, for both cases when $j\neq 0$ or $j=0$,
\begin{equation*}
|c_{j,m,z}|\geq\left|\log^2(r_\ell)-\pi^2 m^2-\theta^2\right|\geq \pi^2\ell.
\end{equation*}
With the same technique, we can prove that when $\pi^2\ell>k^2\|q\|_\infty$,
\begin{equation}\label{eq:int_est_2}
\|w(z,\cdot)\|_{L^2(\Omega_0)}=\|z^{x_1}v_z\|_{L^2(\Omega_0)}\leq \frac{\|f\|_{L^2(\Omega_0)}}{r(\pi^2\ell-k^2\|q\|_\infty)}.
\end{equation}

Now we are prepared to consider the integral $I(r_\ell)$. Let $\epsilon_\ell:=\left({\ell^2+(\ell+1)^2}\right)^{-1/4}$, then
\begin{equation*}
r_\ell\rightarrow0,\quad \epsilon_\ell\rightarrow 0,\quad |\log(r_\ell)|\epsilon_\ell\rightarrow \infty,\quad \text{ as }\ell\rightarrow\infty.
\end{equation*}
Define $\mathcal{C}_1^\ell=\{r_\ell\exp(\i\theta):\,|\theta|\leq\epsilon_\ell\}$ and $\mathcal{C}_2^\ell=\{r_\ell\exp(\i\theta):\,\epsilon_\ell<\theta<2\pi-\epsilon_\ell\}$, then $\mathcal{C}_1^\ell\cup\mathcal{C}_2^\ell=\partial B(0,r_\ell)$. We can write the integral $I(r_\ell)$ as two parts (see Figure \ref{circle}):
\begin{equation*}
\begin{aligned}
\oint_{\partial B(0,r_\ell)}w(z,x)z^{n-1}\d z                    
=&\int_{\mathcal{C}_1^\ell}w(z,x)z^{n-1}\d z+\int_{\mathcal{C}_2^\ell}w(z,x)z^{n-1}\d z\\
=&\i\int_{-\epsilon_\ell}^{\epsilon_\ell} w
\left(r_\ell e^{\i\theta},x\right)r_\ell^n e^{\i n\theta}\d\theta+\i\int_{\epsilon_\ell}^{2\pi-\epsilon_\ell} w\left(r_\ell e^{\i\theta},x\right)r_\ell^n e^{\i n\theta}\d\theta.
\end{aligned}
\end{equation*}

From \eqref{eq:int_est_2}, using Lemma \ref{th:minkowski}, the $L^2(\Omega_0)$-norm of the first term is bounded by
\begin{equation*}
\begin{aligned}
\left\|\i\int_{-\epsilon_\ell}^{\epsilon_\ell} w\left(r_\ell e^{\i\theta},x\right)r_\ell^n e^{\i n\theta}\d\theta\right\|_{L^2(\Omega_0)}&=\left(\int_{\Omega_0}\left|\int_{-\epsilon_\ell}^{\epsilon_\ell}w\left(r_\ell e^{\i\theta},x\right)r_\ell^n e^{\i n\theta}\d\theta\right|^2\d x\right)^{1/2}\\
&\leq \int_{-\epsilon_\ell}^{\epsilon_\ell}\left(\int_{\Omega_0}\left|w\left(r_\ell e^{\i\theta},x\right)r_\ell^n e^{\i n\theta}\right|^2\d x\right)^{1/2}\d\theta\\
&\leq \frac{2 r_\ell^{n-1}\epsilon_\ell}{\pi^2\ell-k^2\|q\|_\infty}\|f\|_{L^2(\Omega_0)}\rightarrow 0
\end{aligned}
\end{equation*}
 as $\ell\rightarrow\infty$, with any fixed integer $n\geq 1$.

For the second term, from \eqref{eq:int_est_1} and Lemma \ref{th:minkowski} again,
\begin{equation*}\left\|\i\int_{\epsilon}^{2\pi-\epsilon} w\left(r_\ell e^{\i\theta},x\right)r_\ell^n e^{\i n\theta}\d\theta\right\|_{L^2(\Omega_0)}
\leq \frac{2\pi r_\ell^{n-1}}{2|\log(r_\ell)\epsilon_\ell|-k^2\|q\|_\infty}\|f\|_{L^2(\Omega_0)}\rightarrow 0
\end{equation*}
 as $\ell\rightarrow\infty$.  Thus
 \begin{equation*}
 \oint_{\partial B(0,r_\ell)}w(z,x)z^{n-1}\d z\rightarrow 0\,\text{ in }L^2(\Omega_0),\quad \ell\rightarrow\infty.
 \end{equation*}
 The proof is finished.
 
\end{proof}

With the above results, we have finally  arrived at the following theorem.

\begin{theorem}\label{th:decomp}
Suppose $RS=\{z_{N+1}^+,z_{N+2}^+,\dots\}$. Then for $n\geq 1$,
\begin{equation}\label{eq:decomp1}
\begin{aligned}
 u(x_1+n,x_2)&=\sum_{j=N+1}^\infty{\rm Res}\left(w(z,x)z^{n-1},z=z_j^+\right)\\
 &+\sum_{j=1}^N\frac{(z_j^+)^{n}\left<q^{-1}f,\psi^+_j(-\i\log(z_j^+),\cdot)\right>}{(\mu^+_j)'(-\i\log(z_j^+))}\psi_j^+(-\i\log(z_j^+),x).
 \end{aligned}
\end{equation}
\end{theorem}

\begin{proof}

The proof of \eqref{eq:decomp1} comes directly from Theorem \ref{th:stab}, \ref{th:generalized_residue}, and Lemma \ref{th:int_lm}. 
\end{proof}

\high{
\subsection{Eigenvalue decomposition of the operator $\mathcal{R}$}

From the last subsection, the solution $u_0$ in $\cup_{n=1}^\infty\Omega_n$ has been written as a linear combination of residues at poles $z_j^+$ for $j\geq N+1$. To study the spectrum decomposition of $\mathcal{R}$, we fix one residue ${\rm Res}(w(z,x),z_j^+)$ where $w(z,\cdot)\in H^1_z(\Omega_0)$ is the solution of \eqref{eq:wg_z1}-\eqref{eq:wg_z2}, and study the property of this function.  For simplicity, define the following operator when $j=N+1,N+2,\dots$:
 \begin{eqnarray*}
 \L_j: \mathcal{P}({\UD}, L^2(\Omega_0))&\rightarrow &H^1(\Omega_0)\\
 g(z,x)=\sum_{\ell=1}^N p_\ell(z)f_\ell(x)&\mapsto&\frac{1}{2\pi\i}\oint_{|z-z_j^+|=\delta_j}(I-\B_z)^{-1}\left[\sum_{\ell=1}^N p_\ell(z)\widetilde{f}_\ell(x)\right] \d z
\end{eqnarray*}
where $p_\ell$ is a polynomial, $\widetilde{f}_\ell\in H^1(\Omega_0)$ is obtained from \eqref{eq:tilde_f}, $\UD:=\{z\in\C:\,|z|\leq 1\}$ is the closed unit disk and the space $\mathcal{P}(\UD, L^2(\Omega_0))$ is defined by
\begin{equation*}
\mathcal{P}(\UD, L^2(\Omega_0)):=\left\{g(z,x)=\sum_{\ell=1}^N p_\ell(z)f_\ell(x):\,N\in\N_+,\,p_\ell(z)\text{ is a polynomial, }f_\ell\in L^2(\Omega_0)\right\}.
\end{equation*}
Note that $RS=\left\{z_j^+:\,j=N+1,N+2,\dots\right\}\subset\UD$. As the elements in this space are linear combinations of productions of polynomials and $L^2$-functions, for any fixed $z\in\UD$, $g(z,\cdot)$ is a well-defined function in $L^2(\Omega_0)$, thus we define the norm in this space by
\begin{equation*}
 \big\|g(z,x)\big\|_{\mathcal{P}(\UD, L^2(\Omega_0))}=\sup_{z\in\UD}\big\|g(z,\cdot)\big\|_{L^2(\Omega_0)}.
\end{equation*}
Then $\L_j$ is bounded from $\mathcal{P}(\UD, L^2(\Omega_0))$ to $H^1(\Omega_0)$. We define the following space
\begin{equation*}
\E_{j}:=\Big\{\L_j (g(z,x)):\, g\in\mathcal{P}(\UD,L^2(\Omega_0))\Big\}\subset H^1(\Omega_0),
\end{equation*}
thus
\begin{equation*}
 \E_{j}=\left\{\sum_{\ell=1}^N{\rm Res}\Big(p_\ell(z)w_\ell(z,x),z=z_j^+\Big):\,w_\ell(z,x)\in H_z^1(\Omega_0) \text{ is the solution of \eqref{eq:wg_z1}-\eqref{eq:wg_z2} with source term $f_\ell$}\right\}.
\end{equation*}
From Theorem \ref{th:ana_fred_thy}, $\E_j$ is a finite dimensional space. 
Let the trace operator $\Upsilon:\, H^1(\Omega_0)\mapsto H^{1/2}(\Gamma_1)$, and $\Upsilon\E_j$ be the space of all the \high{traces on the boundary $\Gamma_1$ of functions in $\E_j$}.


Then we study the property of the space $\E_j$. As $w(z,x)\big|_{\Gamma_{j+1}}=z w(z,x)\big|_{\Gamma_j}$,
\[\mathcal{R}w(z,x)\Big|_{\Gamma_1}=w(z,x)\Big|_{\Gamma_{2}}=z w(z,x)\Big|_{\Gamma_{1}}.\]
With this property, we have
\begin{equation*}
\begin{aligned}
 \mathcal{R}\Big(\Upsilon\L_j (p(z)f)\Big)&=\left.\left(\frac{1}{2\pi\i}\oint_{|z-z_j^+|=\delta_j} (I-\B_z)^{-1} p(z)\widetilde{f}(x)\d z\right)\right|_{\Gamma_{2}}\\
 &={\rm Res}\left((I-\B_z)^{-1}p(z)\widetilde{f}(x),z=z_j^+\right)\Big|_{\Gamma_{2}}\\
 &=z_j^+{\rm Res}\left((I-\B_z)^{-1}p(z)\widetilde{f}(x),z=z_j^+\right)\Big|_{\Gamma_{1}}\\ &={\rm Res}\left((I-\B_z)^{-1}zp(z)\widetilde{f}(x),z=z_j^+\right)\Big|_{\Gamma_{1}}\\&
 =\Upsilon\L_j\Big(zp(z) f\Big).
 \end{aligned}
\end{equation*}
Denote the $m$-th power of $\mathcal{R}$ by $\mathcal{R}^{m}$, then it is defined by
\begin{equation*}
 \mathcal{R}^{m}:\, u\big|_{\Gamma_1}\mapsto u\big|_{\Gamma_{1+m}}.
\end{equation*}
From similar technique, the following equation is satisfied:
\begin{equation*}
 \mathcal{R}^{m}\Big(\Upsilon\L_jp(z) f\Big)=\Upsilon\L_j\Big(z^m p(z)f\Big).
\end{equation*}

From the linearity of the operators $\mathcal{R}^m$, $\Upsilon$ and $\L_j$
\begin{equation*}
\begin{aligned}
\mathcal{R}^m\Big(\Upsilon\L_j g(z,x)\Big)&=\mathcal{R}^m\left(\Upsilon\L_j \left[\sum_{\ell=1}^N p_\ell(z)f_\ell(x)\right]\right)=\sum_{\ell=1}^N \mathcal{R}^m\Big(\Upsilon\L_j(p_\ell(z)f_\ell(x))\Big)\\
&=\sum_{\ell=1}^N\Upsilon\L_j\Big(z^m p_\ell(z)f_\ell(x)\Big)=\Upsilon\L_j\left(\sum_{\ell=1}^N z^m p_\ell(z)f_\ell(x)\right)=\Upsilon\L_j\Big(z^m g(z,x)\Big).
\end{aligned}
\end{equation*}
The results are extended in the following lemma.

\begin{lemma}\label{th:poly_R}
 Let $\widetilde{p}(z)$ be a  polynomial of $z$ with a positive degree. For any $g(z,x)\in \mathcal{P}\left(\UD,L^2(\Omega_0)\right)$,
 \begin{equation*}
  \widetilde{p}\left(\mathcal{R}\right)\Big(\Upsilon\L_j g(z,x)\Big)=\Upsilon\L_j\Big(\widetilde{p}(z)g(z,x)\Big).
 \end{equation*}
\end{lemma}

With this lemma,  we are prepared to study the structure of the set $\E_j$ and the space $X$. Note that the set $\E_j$ is only defined for $j\geq N+1$, first we define the set for $j=1,2,\dots,N$ as:
\[
\E_j:= \big\{a\psi_j^+(-\i\log(z_j^+),\cdot):\,a\in\C\big\},
\]
where $\psi_j^+$ is the eigenfunction defined as in Section 6. Let $\nu^+(j)$ be the Riesz number of $\mathcal{R}-z_j^+I$, then $\nu^+(j)=1$ for $j=1,2,\dots,N$. Now the structures of $\E_j$ and $X$ are concluded in the following theorem.

\begin{theorem}\label{th:generalized_eigenfunction} 
For any integer $j\geq N+1$, $z_j^+\in RS$ is an eigenvalue of $\mathcal{R}$. Moreover, any non-zero element in $\Upsilon\E_j$ is a generalized eigenfunction of $\mathcal{R}$ associated with the eigenvalue $z_j^+$, and $\Upsilon\E_j$ is an invariant subspace of the operator $\mathcal{R}$.

Moreover, we have the following decomposition:
\begin{equation*}
X=\overline{\bigoplus_{j\in\N}\mathcal{N}\left((\mathcal{R}-z_j^+I)^{\nu^+(j)}\right)}
\end{equation*}
and $\mathcal{R}$ has the Jordan normal form in the space $\bigoplus_{j\in\N}\mathcal{N}\left((\mathcal{R}-z_j^+I)^{\nu^+(j)}\right)$, where $\mathcal{N}(P)$ is the kernel of the operator $P$.
\end{theorem}

\begin{proof}
As $z_j^+\in\mathbb{F}$ for any $j\geq N+1$, it is an eigenvalue of $\mathcal{R}$ as there is a Bloch wave solution \high{corresponding} to $z_j^+$. As $(I-\B_z)^{-1}$ is a family of meromorphic operators with respect to $z$ and $z_j^+$ is a pole, it has the Laurent series expansion in a small enough neighbourhood of $z_j^+$:
 \begin{equation}\label{eq:exp_Bz}
  (I-\B_z)^{-1}=\sum_{\ell=-M}^\infty (z-z_j^+)^\ell \B_\ell,
 \end{equation}
where $M>0$ is a positive integer and $\B_\ell$ is bounded. The series converges in a punctured neighborhood $\mathring{B}(z_j^+,\delta_j)$, where $\delta_j>0$ is sufficiently small. From Lemma \ref{th:poly_R},
\begin{equation*}
\begin{aligned}
 (\mathcal{R}-z_j^+ I)^M\Big(\Upsilon\L_j g(z,x)\Big)&=\Upsilon\L_j\left((z-z_j^+)^M \left[\sum_{\ell=1}^N p_\ell(z)f_\ell(x)\right]\right)\\
 &=\sum_{\ell=1}^N\Upsilon\left(\frac{1}{2\pi\i}\oint_{|z-z_j^+|=\delta_j}(z-z_j^+)^M(I-\B_z)^{-1} p_\ell(z) \widetilde{f}_\ell\d z\right)\\
 &=\sum_{\ell=1}^N\Upsilon\left(\frac{1}{2\pi\i}\oint_{|z-z_j^+|=\delta_j}(z-z_j^+)^M\left[\sum_{s=-M}^\infty (z-z_j^+)^s \B_s\right]  p_\ell(z)  \widetilde{f}_\ell\d z\right)\\
 &=\sum_{\ell=1}^N\Upsilon\left(\frac{1}{2\pi\i}\oint_{|z-z_j^+|=\delta_j}\left[\sum_{s=0}^\infty (z-z_j^+)^{s} \left(\B_{s-M}  \widetilde{f}_\ell\right)\right]p_\ell(z) \d z\right).
 \end{aligned}
\end{equation*}
Thus for each $\ell=1,2,\dots,N$, the integrand $(I-\B_z)^{-1}(z-z_j^+)^M p_\ell(z)  \widetilde{f}_\ell$ is analytic at the point $z_j^+$. From Cauchy's integral formula, the integral equals to $0$. Thus
\begin{equation*}
 (\mathcal{R}-z_j^+ I)^M\Big(\Upsilon\L_j g(z,x)\Big)=0,
\end{equation*}
{which} implies that $\Upsilon\L_j g(z,x)$ is a generalized eigenfunction of $\mathcal{R}$ associated with the eigenvalue $z_j^+$.

Note that for any $g\in\mathcal{P}\left(\UD,L^2(\Omega_0)\right)$, $\Upsilon\L_j g(z,x)\in\Upsilon\E_j$. From Lemma \ref{th:poly_R}, {$\mathcal{R}\Big[\Upsilon\L_j g(z,x)\Big]\in\Upsilon\E_j$}. Thus $\Upsilon\E_j$ is an invariant subspace of the operator $\mathcal{R}$ in \high{$H^{1/2}(\Gamma)$}.  This implies that $\Upsilon\E_j\subset \mathcal{N}\left((\mathcal{R}-z_j^+I)^M\right)$\high{, which is the null space of $(\mathcal{R}-z_j^+I)^M$}.

 From Corollary \ref{th:U_co}, $\mathcal{V}=H^{1/2}(\Gamma_1)=X$. From the representation of $u$ in \eqref{eq:LAP} and Theorem \ref{th:generalized_residue_l2}, $\mathcal{V}=\left\{\sum_{j=1}^\infty\Upsilon\L_j f:\,f\in L^2(\Omega_0)\right\}$, thus
\begin{equation*}
X=H^\high{1/2}(\Gamma_1)=\overline{\rm span}\{\Upsilon\E_j:\,j\in\N\}.
\end{equation*}
 From the fact that the generalized eigenfunctions are linearly independent, the subspaces $\left\{\Upsilon\E_j\right\}_{j\in\N}$ are linearly independent and 
\begin{equation*}
X=\overline{\bigoplus_{j\in\N} \Upsilon\E_j}.
\end{equation*}

As $\mathcal{R}$ is a compact operator and $z_j^+$ is an eigenvalue, and $\nu^+(j)$ is the Riesz number for $\mathcal{R}-z_j^+I$, then
\begin{eqnarray*}
&& \emptyset=\mathcal{N}\left((\mathcal{R}-z_j^+I)^0\right)\subset\mathcal{N}\left((\mathcal{R}-z_j^+I)^1\right)\subset\cdots\subset\mathcal{N}\left((\mathcal{R}-z_j^+I)^{\nu^+(j)}\right)  =\mathcal{N}\left((\mathcal{R}-z_j^+I)^{\nu^+(j)+1}\right) ;\\
&& X=\left(\mathcal{R}-z_j^+I\right)^0(X)\supset\left(\mathcal{R}-z_j^+I\right)^1(X)\supset\cdots\supset\left(\mathcal{R}-z_j^+I\right)^{\nu^+(j)}(X) =\left(\mathcal{R}-z_j^+I\right)^{\nu^+(j)+1}(X).
\end{eqnarray*} 
Thus for any fixed $j$,
\begin{equation*}
\Upsilon\E_j\subset \mathcal{N}\left((\mathcal{R}-z_j^+I)^{M}\right)\subset \mathcal{N}\left((\mathcal{R}-z_j^+I)^{\nu^+(j)}\right).
\end{equation*}
Then we have the following decomposition:
\begin{equation*}
X=\overline{\bigoplus_{j\in\N}\mathcal{N}\left((\mathcal{R}-z_j^+I)^{\nu^+(j)}\right)}
\end{equation*}
and $\mathcal{R}$ has the Jordan normal form in the space $\bigoplus_{j\in\N}\mathcal{N}\left((\mathcal{R}-z_j^+I)^{\nu^+(j)}\right)$. The proof is finished.
\end{proof}

 

We have proved that the generalized eigenfunctions of $\mathcal{R}$ form a complete set in $X$. We can also conclude the same results for the translation operator defined on the left boundary of the unit cell. 
}

\subsection{Decomposition of LAP solutions}

With the help of the \high{spectral} decomposition of $\mathcal{R}$, we would like to discuss the decomposition of the LAP solution with the help of the spectrum decomposition. As
\begin{equation*}
 u_0(x_1+n,x_2)=\sum_{j=N+1}^\infty \L_j (z^{n-1}f),
\end{equation*}
the LAP solution is decomposed into infinite number of generalized eigenfunctions \high{corresponding} to $z_j^+$. 

Recall $S_+=S_+^0\cup RS$, $S_+^0=\{z_1^+,\dots,z_N^+\}$ and $RS=\{z_{N+1}^+,\dots\}$, let
\begin{equation*}
 u_j(x_1+n,x_2)= \L_j (z^{n-1}f).
\end{equation*}

As any $z_j^+$ is an isolated singularity, for any $j\in\N$, there is a  $\delta_j>0$ small enough such that  $B(z_j^+,\delta_j)\cap S_+=\left\{z_j^+\right\}$. Thus 
\begin{equation*}
 u_j(x_1+n,x_2)=\frac{1}{2\pi\i}\oint_{|z-z_j^+|=\delta_j}(I-\B_z)^{-1}z^{n-1} \widetilde{f}\d z.
\end{equation*}
 As $(I-\B_z)^{-1}$ exists and is uniformly bounded for $|z-z_j^+|=\delta_j$, there is a $C=C(\delta_j)>0$ such that
\begin{equation*}
 \left\|(I-\B_z)^{-1}f\right\|_{H^1(\Omega_0)}\leq C\|f\|_{L^2(\Omega_0)},\quad\high{\text{for $|z-z_j^+|=\delta_j$}}.
\end{equation*}
Then for any $n\geq 1$, from Lemma \ref{th:minkowski},
\high{
\begin{equation*}
\begin{aligned}
 \left\|u_j(x_1+n,x_2)\right\|_{L^2(\Omega_0)}&=\left\|\frac{1}{2\pi\i}\oint_{|z-z_j^+|=\delta_j}(I-\B_z)^{-1}z^{n-1} \widetilde{f}\d z\right\|_{L^2(\Omega_0)}\\
 &=\left\|\frac{\delta_j}{2\pi}\int_0^{2\pi}(I-\B_{z_j^++\delta_j e^{\i\theta}})^{-1}\left(z_j^++\delta_j e^{\i\theta}\right)^{n-1} \widetilde{f}e^{\i\theta}\d \theta\right\|_{L^2(\Omega_0)}\\
 &=\left(\int_{\Omega_0}\left|\frac{\delta_j}{2\pi}\int_0^{2\pi}(I-\B_{z_j^++\delta_j e^{\i\theta}})^{-1}\left(z_j^++\delta_j e^{\i\theta}\right)^{n-1} \widetilde{f}e^{\i\theta}\d \theta\right|^2\d x\right)^{1/2}\\
 &\leq \frac{\delta_j}{2\pi}\int_0^{2\pi}\left(\int_{\Omega_0}\left|(I-\B_{z_j^++\delta_j e^{\i\theta}})^{-1}\left(z_j^++\delta_j e^{\i\theta}\right)^{n-1} \widetilde{f}\right|^2\d x\right)^{1/2}\d z\\
 &\leq C\|f\|_{L^2(\Omega_0)}\left(\int_0^{2\pi}|z_j^++\delta_j e^{\i\theta}|^{n-1}\d z\right)\\
 &\leq C\|f\|_{L^2(\Omega_0)}\left(\delta_j+|z_j^+|\right)^{n-1}.
 \end{aligned}
\end{equation*}}
Thus we conclude the result in the following theorem.

\begin{theorem}
 For any $j\in\N$, given a small enough $\delta>0$, there is a constant $C=C(\delta)>0$ such that for any $n\geq 1$ and $j\geq N+1$,
 \begin{equation*}
  \|u_j\|_{H^1(\Omega_n)}\leq C\left(\delta+\left|z_j^+\right|\right)^{n-1}.
 \end{equation*}
 Thus for any $z_j^+\in RS$, $u_j$ decays exponentially when $x_1\rightarrow+\infty$.  
\end{theorem}

From the theorem above, the mode $u_j$ for $j\geq N+1$ is evanescent. From \eqref{eq:LAP}, the solution is composed of finite number of rightward propagating Bloch wave solutions, and an infinite number of generalized eigenfunctions that are evanescent.

\section*{Acknowlegdments}
The work  was supported by Deutsche Forschungsgemeinschaft (DFG) through CRC 1173.

\bibliographystyle{alpha}
\bibliography{ip-biblio} 

\end{document}